\documentclass[reqno,oneside,12pt]{amsart}

\usepackage[T1]{fontenc}
\usepackage{times,mathptm,caption}
\usepackage{amssymb,enumerate}

\usepackage{hyperref}
\hypersetup{
    colorlinks,
    citecolor=black,
    filecolor=black,
    linkcolor=red,
    urlcolor=black
}

\theoremstyle{plain}

\newtheorem{thm}{Theorem}[section]
\newtheorem{cor}[thm]{Corollary}

\newtheorem{lem}[thm]{Lemma}
\newtheorem{proposition-principale}[thm]{Proposition principale}
\newtheorem{thm-principal}{Main Theorem}

\newtheorem{mthm}{Theorem}

\theoremstyle{definition}

\newtheorem{eg}[thm]{Example}
\newtheorem{rem}[thm]{Remark}

\newenvironment{defi-G}
{\noindent{\bf Definition.}\it}{\\}

\newenvironment{thm-M}
{\noindent{\bf Main Theorem.}\it }{}

\newenvironment{thm-C}
{\noindent{\bf Classification Theorem.}\it }{}

\newenvironment{thm-A}
{\noindent{\bf Theorem A.}\it}{\\ }

\newenvironment{thm-B}
{\noindent{\bf Theorem B.}\it}{\\ }

\newenvironment{thm-BB}
{\noindent{\bf Theorem B'.}\it}

\def\C{\mathbf{C}}
\def\R{\mathbf{R}}

\def\Z{\mathbf{Z}}

\def\Fp{{\mathbf{F}}_p}
\def\Fq{{\mathbf{F}}_q}
\def\bfk{{\mathbf{k}}}

\def\F{{\mathsf{F}}}

\def\Aut{{\sf{Aut}}}
\def\Out{{\sf{Out}}}
\def\Inn{{\sf{Inn}}}

\def\haar{{\mathrm{Haar}}}
\def\mes{{\mathrm{m}}}

\def\Rep{{\sf{Hom}}}
\def\Mod{{\sf{Mod}}}
\def\Epi{{\sf{Epi}}}
\def\Epitop{{\sf{Epi}}} 
\def\Epifin{{\sf{Epi}^{\#}}}
\def\Red{{\sf{Red}}}

\def\Redtop{{\sf{Red}}}

\def\card{{\mathrm{card}}}
\def\inc{{\mathrm{ic}}}
\def\lc{{\mathrm{cl}}}
\def\rg{{\mathrm{d}}}

\def\GL{{\sf{GL}}}
\def\SO{{\sf{SO}}}
\def\OO{{\sf{O}}}

\def\UU{{\sf{U}}}

\def\Bij{{\sf{S}}}
 
\def\SL{{\sf{SL}}\,}
\def\Sp{{\sf{Sp}}\,}

\def\tr{{\sf{tr}}}

\def\det{{\sf{det}}}

\newcommand{\Id}{{\rm Id}}

\newcommand{\grouptop}[1]{\left< #1 \right>_{top}}

\addtocounter{section}{0}             
\numberwithin{equation}{section}       

\begin{document}

\setlength{\baselineskip}{0.56cm}        
%
%
\title[$\Aut(\F_n)$-action on $\Rep(\F_n;G)$]
{Rigidity of the dynamics of $\Aut(\F_n)$ on representations into  a compact group}
\date{2025/2026}
\author{S. Cantat, C. Dupont, F. Martin-Baillon}
\address{CNRS, IRMAR - UMR 6625 \\ 
Universit{\'e} de Rennes 
\\ France}
\email{serge.cantat@univ-rennes.fr}
\email{christophe.dupont@univ-rennes.fr}
\email{florestan.martin-baillon@mis.mpg.de}
\thanks{{\footnotesize{The research activities of the authors are partially funded by the European Research Council (ERC GOAT 101053021 and  ERC PosLieRep 101018839).}}
}

%
%

%
%

%
%

\begin{abstract}
Let $G$ be a compact Lie group. Let $\F_n$ be the free group of rank~$n$. 
We describe the orbits of $\Aut(\F_n)$ on $\Rep(\F_n;G)$ when $n$ is sufficiently large. 
The dynamics stabilizes: orbit closures and invariant probability measures are algebraic, as in Ratner's 
theorems. 

\vspace{0.15cm}

\noindent{\sc{R\'esum\'e.}} Soit $G$ un groupe de Lie compact. Soit $\F_n$ le groupe libre sur~$n$ lettres. 
Nous  décrivons les orbites de l'action du groupe $\Aut(\F_n)$ sur $\Rep(\F_n;G)$ pour $n$ suffisamment grand. 
Une forme de stabilité apparaît, les adhérences d'orbites et les mesures de probabilité invariantes   étant algébrique, comme dans les théorèmes de Ratner. 
 \end{abstract}

\maketitle

\setcounter{tocdepth}{1}

\tableofcontents

\vfill




\pagebreak

\section{Introduction}

\subsection{Nielsen moves} Let $ \F_n=\langle a_1, \ldots, a_n\, \vert \, \emptyset\rangle$ be the free group on $n$ generators. 
Given some group $G$, we denote by $\Rep(\F_n; G)$ the set of homomorphisms $\F_n\to G$. 
A homomorphism 
$\rho\colon \F_n\to G$ is uniquely determined by the tuple $(\rho(a_1), \ldots, \rho(a_n))$, which provides an identification $\Rep(\F_n;G)=G^n$.

The group of automorphisms of $\F_n$ acts on $\Rep(\F_n;G)$ by precomposition; more precisely, if $\varphi\in \Aut(\F_n)$ and $\rho\in \Rep(\F_n;G)$, then $\varphi_*\rho=\rho\circ \varphi^{-1}$. Under the identification $\Rep(\F_n;G)\simeq G^n$, this action coincides with the action of {\bf{Nielsen moves}} that is, with the action generated by the group of permutations on $n$ letters and  the following elementary transformations

\begin{align}
	\label{eq:elementary}
 (g_1 , \dots , g_i  , \dots ,  g_n ) &\mapsto (g_1 , \dots , g_i^{-1}  , \dots , g_n ) \\
(g_1 , \dots , g_i  , \dots ,  g_n ) &\mapsto (g_1 , \dots , g_i g_j  , \dots , g_n )  
 \end{align}
for pairs $(i,j) \in \{ 1 , \dots , n \}$ with $i \neq j$. 
Our goal is to understand this action when $G$ is a compact Lie group: we shall see that the main dynamical features of this action ``stabilize'' as $n$ becomes large. 

\subsection{Stabilization of the dynamics}
\label{par:intro_main_result}
  To state our main result we need the following notation.
Let $H$ be a closed subgroup of $G$. Then $H$ is a Lie subgroup of $G$; we denote by $H^\circ$ its neutral component
 and by $ H^{\#} $ the finite group $H/H^\circ$. 
Let $\pi : H \to H^{\#}$ be the projection
and let $\# \colon \Rep(F_n; H) \to \Rep(F_n; H^{\#})$ be defined by $\rho \mapsto \rho^\# := \pi \circ \rho$.
Define $ \Epifin(\F_n; H) $
to be the subset of representations
$ \rho : F_n \to H $ such that the induced representation
$ \rho^{\#} : \F_n \to H^{\#} $
is surjective.
Then $\Epifin(\F_n;H)$ is a closed $\Aut(\F_n)$-invariant subset of $\Rep(\F_n;G)$.

There is a natural probability measure $\mes_H^n$ on
$\Epifin(\F_n;H)$,
invariant by $ \Aut(\F_n) $,
constructed in the following way.
Firstly, if $ H $ is connected,
we normalize its Haar measure $\haar_H$ to be a probability measure and define $\mes_H^n = \haar_H^{\otimes n}$
on $\Epifin(\F_n;H) = \Rep(\F_n;H)\simeq H^n$. It is invariant under Nielsen moves (resp.\ 
under $\Aut(\F_n)$) because a compact Lie group is unimodular, so that $h\mapsto hg$ and $h\mapsto h^{-1}$ preserve $\haar_H$.
Secondly, if $ H $ is finite, we define $ \mes_H^n $ to be the uniform measure on $ \Epifin(F_n; H) $.
For the general case,
note that if $F\subset \Rep(\F_n; H)=H^n$ is a fiber of $\#$, and $(g_1, \ldots, g_n)$ is a point of $F$, 
then the multiplication by $(H^\circ)^n$ on the right provides a natural diffeomorphism 
\begin{equation}\label{eq:mes_H}
F=\prod_{i=1}^n g_i H^\circ. 
\end{equation}
Moreover, such a fiber is contained in $\Epifin(F_n; H) $ if and only if it intersects it non-trivially.
We can now define $ \mes_H^n $ to be the measure on $ \Epifin(F_n; H) $
such that its projection to $ \Epifin(F_n; H^{\#})$ is $ m_{H^{\#}} $
and its restriction to each fiber of the projection $\Epifin(F_n; H) \to \Epifin(F_n; H^\#) $ is proportional to
$ \mes_{H^{0}} $ under the diffeomorphism given by Equation~\eqref{eq:mes_H}. By unimodularity, this definition does not depend on the chosen base points $(g_1, \ldots, g_n)$ in the fibers of $\#$ and gives a probability measure $ \mes_{H}^n $ that is invariant under the action of $\Aut(\F_n)$.

\begin{mthm}
	\label{thm:main}
	Let  $G$ be a compact Lie group. If $n$ is large enough, then 
	\begin{enumerate}[\rm (1)]
		\item for any $\rho\in \Rep(\F_n;G)$, the
			closure of $\Aut(\F_n)\rho$ in $\Rep(\F_n;G)$ is equal to $\Epifin(\F_n;H_\rho)$, where $H_\rho$ is the closure of the group $\rho(\F_n)$;
		\item the closures of orbits of $\Aut(\F_n)$ in $\Rep(\F_n;G)$ are exactly the 
			$\Epifin(\F_n;H)$, for the closed subgroups $H$ of $G$;
		\item if $\mu$ is an $\Aut(\F_n)$-invariant and ergodic probability measure on
			$\Rep(\F_n; G)$,  there exists a unique closed subgroup $H\subset G$ such that
			$\mu$ coincides with~$\mes_H^n$.
	\end{enumerate}
\end{mthm}

The upshot is that the dynamics stabilizes for $n$ large: orbit closures and invariant measures become algebraic,  as in Ratner's results in homogeneous dynamics.
We shall give a precise meaning to the assumption ``{\sl{$n$ is large enough}}''. For instance, there is a universal constant $b\leq 10^3$ such that
\begin{equation}
\label{eq:N(m)_intro}
n\geq 2 m^2\log_2(m)+3m^2+b
\end{equation}
is sufficient, where $m$ is the dimension of the smallest faithfull linear representation of $G$. See~Remark~\ref{rem:estimation_Nm} below. 
A similar phenomenon already appeared in the work of Avni and Garion~\cite{Avni-Garion:2008} on the product replacement algorithm in finite simple groups.

\subsection{Redundancy} To prove such a result, we follow the same path as Gelander in~\cite{Gelander:Aut(Fn)2008}. Our main technical input concerns redundant homomorphisms $\rho\colon \F_n\to G$ or, equivalently, redundant tuples $(g_1, \ldots, g_n)\in G^n$, a concept that we now introduce. 
For $A \subset G$, denote by $\langle  A \rangle$ the subgroup generated by
$A$ and by $\langle  A \rangle_{top}$ the closure of $\langle  A \rangle$ in $G$. 
Let $\Epitop(\F_n ; G)$ be the set of homomorphisms with $\langle  \rho(\F_n) \rangle_{top}=G$, i.e.\  topological epimorphisms.
We have $$\Epitop(\F_n ; G) \subset \Epifin(\F_n ; G)  \subset \Rep(\F_n;G) .$$
Following Lubotzky, Gelander and Minsky \cite{Lubotzky, Gelander-Minsky}
we say that $\rho$ is {\bf{redundant}}
if there is a decomposition $\F_n=A\star B$ into a non-trivial free product (i.e.\ $A$ and $B$ are free groups of positive ranks) such that $\rho(A)$ and $\rho(\F_n)$ have the same closure; equivalently, $\rho$ is redundant if there is a $\varphi\in \Aut(\F_n)$ such that  
\[
\langle 
\rho(\varphi^{-1}(a_1)), 
\ldots, 
\rho(\varphi^{-1}(a_{n-1}))
\rangle_{top}
= \overline{\rho(\F_n)}.
\]
So, being redundant is a topological notion. Note that, even when $G$ is finite, it does not coincide with the notion used by Lubotzky in~\cite{Lubotzky}, for his notion concerns $\rho$ while our notion concerns the orbit of $\rho$ under $\Aut(\F_n)$.
The homomorphism  is 
{\bf{$k$-redundant}} if we can impose the rank of $B$ to be $\geq k$; so, being redundant is the same as being $1$-redundant. 

Translated into properties of tuples $(g_1, \ldots, g_n)\in G^n$, $k$-redundancy means
that, after some Nielsen move mapping the $g_i$ to some $g_i'$, one can forget
the last $k$ generators $g_{n'-k+1}$, $\ldots$, $g_n'$ without changing the
closure of the group generated by the $g_i$. We denote by  $\Redtop(\F_n ; G)$
the set of redundant homomorphisms and by $\Redtop^k(\F_n ; G)$ the set of $k$-redundant homomorphisms. 

\begin{mthm}
\label{THMB}
For every integer $m\geq 1$, there is an integer $N(m)\geq 1$ such that the following property holds for all $n\geq N(m)$: 
if $G$ is any compact subgroup of $\GL_m(\C)$,  then  $\Rep(\F_n ; G) \subset \Redtop(\F_n ; G)$.
\end{mthm}

This statement implies $k$-redundancy for $n\geq N(m)+k-1$. 
To prove Theorem~\ref{THMB}, we rely on (1) Jordan's theorem on the structure of finite subgroups of $\GL_m(\C)$, (2) a result of Borel and Serre which is used to provide an extension of Jordan's theorem to algebraic subgroups of $\GL_m(\C)$, (3) the fact that every compact subgroup of $\GL_m(\C)$ is an algebraic subgroup in the sense of real algebraic groups, and (4) some basic facts concerning finite groups and Wiegoldt's problem. Thus, the proof relies on well known classical results and will not be a suprise to specialists. 
In a work to appear soon, Cohen and Vigdorovich proved a similar theorem using a different strategy (see Remark~\ref{rem:Cohen-Vigdorovich}).

Theorem~\ref{THMB} can be used to prove Theorem~\ref{thm:main} for, by the Peter-Weyl theorem, any compact Lie group embeds in some $\GL_m(\C)$. Indeed, in complement to~\cite{Gelander:Aut(Fn)2008, Gelander-Minsky}, we shall see that enough redundancy implies dynamical rigidity.

Theorems~\ref{thm:main} and~\ref{THMB} answer positively Conjecture~2.6 and Question~2.9 of~\cite{Gelander:JournalAlgebra2024} in the regime when the rank $n$ of $\F_n$ is large enough.
Section~\ref{par:Applications} contains three applications to the dynamics of $\Aut(\F_n)$ on the character variety $\chi(\F_n;\SL_d(\R))$: Corollary~\ref{cor:minimal} describes minimal invariant compact subsets.

An interesting problem is to classify stationary measures. 
For this, 
let $Z\subset \Rep(\F_n; G)$ be the set of representations such $\rho(\F_n)$ is not dense in~$G$.
If in average the dynamics of $\Aut(\F_n)$ on $\Rep(\F_n; G)\setminus Z$ is expanding and $Z$ is repelling, Theorem~1.1.5 of~\cite{BEFRH} implies that every ergodic stationary measure $\mu$ on $\Rep(\F_n; G)$ with $\mu(Z)=0$ is in fact invariant, and then Theorem~\ref{thm:main} shows that $\mu=\mes_G^n$. So, a natural continuation of this paper is to study expansion properties of $\Aut(\F_n)$ on the tangent space of $\Rep(\F_n; G)$.

\subsection*{Acknowledgement} We thank Alonso Beaumont Llona, Sébastien Gouëzel, Vincent Guirardel, Seung uk Jang, François Maucourant for useful discussions. Special thanks to Tsachik Gelander for discussions and the reference~\cite{Avni-Garion:2008} and Alex Lubotzky for connecting us with Tal Cohen and Itamar Vigodrovich. 

\section{Finite groups}

\subsection{Rank, compressibility, and chains} Let $G$ be a finite group. 
We define 
\begin{align*}
\rg(G) &= \min \{ \card(S) \; ; \; S\subset G \text{ and generates } G \}, \\
\tilde \rg(G) &= \max\{ \rg(H)  \; ; \; H \text{ subgroup of } G \},
\end{align*}
The number $\rg(G)$ is called the {\bf{rank}} of $G$. 
A subset $S\subset G$ is {\bf{compressible}} if there is an element $s\in S$ such that  $s\in \langle S\setminus \{s\}\rangle$. If there is no such element $s$, we say that $S$ is incompressible.  
Define two incompressibility constants by
\begin{align}\label{eq: incomp}
\inc(G) &= \max\{ \card(S) \; ; \; S\subset G \text{ generates } G \text{ and is incompressible }\}, \\
\tilde\inc(G) &= \max\{ \card(S) \; ; \; S\subset G \text{ is incompressible}\},
\end{align}
so that $\tilde\inc(G)$ is the maximum of $\inc(H)$ over all subgroups $H\subset G$. We have 
$\rg(G)\leq \inc(G)< \card(G)$ and $\tilde\rg(G)\leq \tilde\inc(G) < \card(G)$.
Let $\lc(G)$ be the maximal length of an increasing chain of subgroups 
\begin{equation}
\label{eq:chain_length}
\{1\} \subsetneq G_1 \subsetneq G_2\subsetneq \cdots \subsetneq G_i\subsetneq\cdots  \subsetneq G_{\ell}=G.
\end{equation}
In a chain of subgroups, the index of $G_i$ in $G_{i+1}$ is at least $2$. Thus,
\begin{equation}
\label{eq:d_ic_lc}
\rg(G)\leq \inc(G)\leq \tilde\inc(G) \leq \lc(G)\leq \log_2(\card(G))
\end{equation}

\begin{eg}
Consider $\Bij_n$, the group of permutations on $n$ letters. By the Stirling formula, $\log(\card(\Bij_n))\simeq n\log(n)$ up to an error of order $\log(n)$; in particular, $n\simeq\log(\card(\Bij_n))/\log\log(\card(\Bij_n))$. Since the groups $\Bij_n$ form an increasing sequence, we obtain $\lc(\Bij_n)\geq n-1$. In fact,  it is proven in~\cite{Cameron-Solomon-Turull:1989} that
$ \lc(\Bij_n) = \lfloor \frac{3n-1}{2}\rfloor - b(n)$ where $b(n)\leq \log_2(n)$ is the number of occurences of $1$ in the binary expansion of $n$.
Since the transpositions $(i,i+1)$ form an incompressible set, $\inc(\Bij_n)\geq n-1$;  Whiston shows in~\cite{Whiston:2000} that this is optimal: if $S\subset \Bij_n$ is incompressible, then  $\card(S)\leq n-1$ and  if $\card(S)=n-1$ then $\langle S\rangle = \Bij_n$.
Thus,  $\inc(\Bij_n) = \tilde\inc(\Bij_n)=n-1$ for every subgroup $H$ of $\Bij_n$. 
\end{eg}

\begin{eg}
\label{eg:abelian_inc_rg}
In the group $\Z/N\Z$, where $N=p_1\cdots p_k$ for $k$ distinct  prime numbers, the set $S=\{N/p_1, \ldots, N/p_k\}$ is generating and incompressible. For instance $(2,3)$ is an incompressible generating pair for $\Z/6\Z$. 
A Nielsen move maps it to $(2,1)$, which is compressible; so, the set of incompressible tuples is not invariant under Nielsen moves. 
\end{eg}

\subsection{Transitivity} When $G$ is finite, we have $$\Epitop(\F_n ; G) =   \{ \rho : \F_n \to G  \; ; \; \rho(\F_n) = G \} ,$$ 
$$ \Red(\F_n ; G) =   \{ \rho : \F_n \to G  \; ; \; \rho(A) = \rho(\F_n) \text{ for some proper free factor } A \text{ of } \F_n  \} . $$

\begin{thm}\label{thm:lubotzky}
Let $G$ be a finite group. 
\begin{enumerate}[\rm (1)]
\item If $n\geq \inc(G)+\rg(G)$ then $\Aut(F_n)$ acts transitively on $\Epi(\F_n;G)$;
\item if $G$ is solvable and  $n>\rg(G)$,  $\Aut(\F_n)$ acts transitively on $\Epi(\F_n;G)$.
\end{enumerate}
Thus, if $n\geq \tilde\inc(G)+\tilde\rg(G)$, the orbits of $\Aut(\F_n)$ in $\Rep(\F_n;G)$ are exactly the subsets $\Epi(\F_n; H)$, where $H$ is any subgroup of $G$. 
\end{thm}

This theorem is proven in~\cite{Lubotzky}, the most difficult part being Assertion~(2), due to Dunwoody~\cite{Dunwoody:1970}.
Assertion~(1) also holds if $n \geq \lc(G) +1$, see~\cite{Lubotzky}. 

We reproduce the proof of Assertion~(1) ({\sl{i.e.}}\ Proposition 3.1 in \cite{Lubotzky}), because it will be generalized in Section \ref{sec:orbit_closure}. We shall also prove Assertion~(2) for abelian groups in Example \ref{eg:zpr}. This is the only case we shall need in this paper, precisely for Lemma \ref{lem: exsequence}.
	
\begin{proof}[Proof of {\rm{(1)}}]
		We shall use the following elementary but crucial remark. 
		\begin{rem}\label{Delta}
	if
	$ (g_1, \dots, g_n) \in G^{n} $
	and $ h \in \langle  g_1, \dots, g_{n-1}\rangle$, then
	there is a Nielsen move
	from 
	$ (g_1, \dots, g_n)$
	to
	$ (g_1, \dots, g_{n-1},  g_n h^{\pm 1} )$
	(resp.\ to $ (g_1, \dots, g_{n-1}, h^{\pm 1} g_n )$).
\end{rem}
		Set
	$ d = \rg(G) $ and 
	$ m = \inc(G) $. We can assume $G$ non-trivial, {\sl{i.e.}}\  $\rg(G)\geq 1$.
	Fix a system 
	$ (g_1, \dots, g_d) $ of generators of $ G $.
	Pick
	$ \rho \in \Epi(F_n; G) $,
	and identify it 
	with the tuple 
	$ \left( h_1, \dots, h_n \right) $
	in $G^n$ 
	defined by $h_i= \rho(a_i)$.
	
	By definition of $ \inc(G) $
	and because $ n \ge \inc(G) + \rg(G) $,
	the set
	$ \left( h_1, \dots, h_n \right) $
	is compressible.
	This means that, up to permutation,
	$ h_n \in \langle h_1, \dots, h_{n-1} \rangle $;
	hence
	$ (h_1, \dots, h_{n-1}) $
	still generates $ G $.
	By induction,
	we can assume that
	$ (  h_1, \dots, h_{m} ) $
	generates $ G $, and consequently
        the remaining elements
	$ h_{m+1}, \dots, h_n $
	are in 
	$ \langle h_1, \dots, h_{m} \rangle $.
	By Remark \ref{Delta}, there is a Nielsen move mapping	
        $ (h_1, \dots, h_{n}) $
	to
	$ (h_1, \dots, h_{m}, 1, \dots, 1) $.
	In the same way,
	$g_1$, $\dots$, $g_d$ 
	are all in 
	$ \langle h_1, \dots, h_{m} \rangle $
	and
	$ (h_1, \dots, h_{m}, 1, \dots, 1) $
	can be moved to the tuple
	\[ (h_1, \dots, h_{m},
	g_1, \dots, g_d,
	1, \dots, 1)
	\]
	because
	$ n \ge m + d $.

	We now go the other way arround.
	The elements 
	$ h_i$ 
	are all in 
	$ \langle g_1, \dots, g_d \rangle$,
	thus
	we can move
	$ (h_1, \dots, h_{m},
	g_1, \dots, g_d,
	1, \dots, 1)
	$
	to
	$ (
	g_1, \dots, g_d,
	1, \dots, 1)
	$.
	This establishes
	that 
	$ (
	g_1, \dots, g_d,
	1, \dots, 1)
	$ is in the orbit
	of $\rho \in \Epi(F_n; G) $,
	and shows that the action
	is transitive on $\Epi(F_n; G)$.
\end{proof}

 If $A$ is an abelian group of rank $d$, the rank of any of its subgroups its at most $d$. Thus, 
 Assertion~(2) of Theorem~\ref{thm:lubotzky} gives

\begin{cor}
\label{cor:lubotzky}
Let $A$ be a finite abelian group. 
\begin{enumerate}[\rm (1)]
\item ${\tilde{\rg}}(A)=\rg(A)$; 
\item if $n>\rg(A)$, the orbits of the action of $\Aut(F_n)$ on $\Rep(\F_n;A)$ are exactly the subsets $\Epi(\F_n;H)$, where $H$ is any subgroup of $A$. 
\end{enumerate}
\end{cor}

On the other hand,  $\inc(A)$, hence $\lc(A)$, is not bounded in terms of $\rg(A)$, as Example~\ref{eg:abelian_inc_rg} shows. To be self-contained, we present Dunwoody's proof of Corollary~\ref{cor:lubotzky}\,(2) (the proof for solvable groups, as in Theorem~\ref{thm:lubotzky}, is similar). 

\begin{proof}[Proof of Corollary~\ref{cor:lubotzky}\,(2)] 
We set $d=\rg(A)$. It suffices to show that the Nielsen moves act transitvely on the systems of generators $(g_1, \ldots, g_n)$.
We argue  by induction on the chain length $\lc(A)$. 
This is trivial when $\lc(A)=0$. 
Now, suppose $\lc(A)=\ell$ and the statement proven up to length $\ell-1$. Fix a maximal sequence of subgroups $G_i\subset A$, as in~\eqref{eq:chain_length}. Then $G_1$ is a proper subgroup of $A$ isomorphic to $\Z/p\Z$ for some prime $p$ and $\lc(A/G_1)=\ell-1$. Let $(h_1, \ldots, h_d)$ be a fixed system of generators of $A$ and  $(g_1, \ldots, g_n)$ be any system of generators. Since $\lc(A/G_1)=\ell-1$, there is a Nielsen move mapping $(g_1, \ldots, g_n)$ to $(m_1h_1, \ldots, m_dh_d, m_{d+1}, \ldots, m_n)$, with each $m_i$ in $G_1$. If $m_j\neq 1$ for some $j\geq d+1$, then $m_j$ generates $G_1$, the last $n$-tuple can be moved to $(h_1, \ldots, h_d, m_j, 1, \ldots, 1)$, and then to $(h_1, \ldots, h_d, 1, \ldots 1)$ since the $h_i$ generate $A$. Otherwise, the $m_ih_i$ generate $A$, so we can replace $m_{d+1}$ by a non-trivial element of $G_1$ and apply the previous argument. In all cases $(h_1, \ldots, h_d, 1, \ldots 1)$ is in the orbit of $(g_1, \ldots, g_n)$, which concludes the proof.
\end{proof}

The next examples show that Corollary~\ref{cor:lubotzky} is optimal, 
and that Theorem~\ref{thm:main} fails when $n$ is too small.

\begin{eg} \label{eg:zpr}
Fix a prime $q$ and an integer $r\geq 1$,  and set $A=\Fq^r$.  By the previous corollary, $\Aut(\F_n)$ acts transitively on $\Epi(\F_n;A)$ if $n\geq r+1$. If $n=r$, the set $\Epi(\F_r; A)$ can be identified to the set of basis of the $\Fq$-vector space $\Fq^r$, hence to $\GL_r(\Fq)$. With such an identification, the action of  $\Aut(\F_r)$ is the action of  $\SL_r^{\pm}(\Fq)$ on $\GL_r(\Fq)$ by right multiplication, where $\SL_r^{\pm}(\Fq)=\{g\in \GL_r(\Fq)\; ; \; \det(g)=\pm 1\}$. This action is not transitive, it has $(q-1)/2$ orbits given by the different values of the determinant on $\GL_r(\Fq)$ up to $\pm 1$. On the other hand, by Corollary~\ref{cor:lubotzky}\,(2), the action of $\Aut(\F_n)$ is transitive on $\Epi (\F_n ;  A)$  if $n\geq r+1$
(\footnote{This can be proved directly by linear algebra. Indeed, if $(g_1, \ldots, g_n)$ generates $A$, one can assume that $(g_1, \ldots, g_r)$ is a basis and  find a matrix in $\SL_r(\Fq)$ that transforms the $g_i$ into $(e_1, \ldots, \alpha e_r, g'_{r+1}, \ldots, g'_n)$, where $(e_1, \ldots, e_r)$ is the standard basis, $\alpha\in \Fq^\times$, and the $g'_j$ are in $A$. Thus, with additional Nielsen moves, one can replace $(g'_{r+1}, \ldots, g'_n)$ by $(e_r, 0, \ldots, 0)$, and then arrive at $(e_1, \ldots, e_r, 0, \ldots, 0)$.}).
\end{eg}

\begin{eg} Fix two integers $r, q\geq 1$.
Consider the lamplighter group $L$ obtained by letting $\Z/r\Z$ act on the additive group of functions from $\Z/r\Z$ to $\Z/q\Z$. 
Thus, $L$ is the semi-direct product of $\Z/r\Z$ and the abelian group $F$ of all functions $f\colon \Z/r\Z\to \Z/q\Z$, where $n\in \Z/r\Z$ acts on $F$ by $f(\cdot)\mapsto f(\cdot-n)$. 
Its rank is $2$: indeed, $L$ is generated by  $1\in \Z/r\Z$ and the Dirac function $\delta_0\in F$ ({\sl{i.e.}}\ $\delta_0(0)=1$ and $\delta_0(n)=0$ for all $n\in \Z/r\Z\setminus \{ 0\}$). Dunwoody's theorem shows that  $\Aut(\F_n)$   acts transitively on  $\Epi(\F_n; L)$ when $n\geq 3$. On the other hand, the subgroup $F\subset A$ is isomorphic to $(\Z/q\Z)^r$, so its rank is  $r$. By, the previous example, the action of $\Aut(\F_n)$ on $\Epi(\F_n;F)$ is not transitive when  $n\leq q$; if $q$ is prime and  $n=r$, there are  $(q-1)/2$ orbits.
\end{eg}

\subsection{Jordan families}
If $G$ is a finite group and $N$ is a normal subgroup of $G$, 
\begin{enumerate}[\rm (1)]
\item $\inc(G)\leq \inc(G/N)+\tilde\inc(N)$;
\item $\lc(G)= \lc(G/N)+\lc(N)$.
\end{enumerate}
These results are taken from~\cite{Cameron-Solomon-Turull:1989, Dunwoody:1970, Whiston:2000}.
We only sketch the proof of the first assertion, because it is useful to have it in mind in what follows. 
For this, fix some incompressible generating set $S\subset G$, let $n =\card (S)$. Its projection $\overline S$ in $G/N$ generates $G/N$, hence there is $T\subset S$ such that $\overline T$ generates $G/N$ and $\card(T) \leq \inc(G/N)$. Permuting the elements of $S$, we can write
$$S=\{ t_1, \ldots, t_d, s_{d+1}, \ldots, s_n \} $$
with $T=\{ t_1, \ldots t_d \}$. There are words $w_j$ in the $t_i$ such that the $ s'_j:=s_jw_j$ are in $N$ for every $j>d$. 
Let $S'=\{ t_1, \ldots, t_d, s'_{d+1}, \ldots, s'_n \}$ be the set obtained by these Nielsen moves. Then, $\card(S')=\card(S)$, $S'$ generates $G$, and $\{ s'_{d+1}, \ldots, s'_{n} \}$ is an incompressible system in $N$; indeed, if one of the $s'_k$ could be generated by the other $s'_j$,
we could write $s_kw_k$ as a word in the $s_jw_j$, hence $s_k$ itself would be a word in the $t_i$ and the $s_j$ (for $d+1 \leq j \leq n$ and $j \neq k$) contradicting the incompressibility of~$S$. Thus, $\inc(G)\leq \inc(G/N)+\tilde\inc(N)$. 

\begin{lem} 
\label{lem: exsequence} 
Let $G$ be a finite group and let $ 0\to A \to G \to Q\to 1$
be an exact sequence 
with $A$ abelian. If  \[ n \geq \rg(A) + \tilde\inc(Q) +1, \] every $\rho \in \Rep(\F_n;G)$ is redundant.
\end{lem}

\begin{proof}
Let $\rho$ be an element of $\Rep(\F_n;G)$ and let $H$ be its image. 
Denote by $Q_H$ the projection of $H$ in $Q$ and by $A_H$ the intersection of $H$ with $A$.
Set   $$d=\rg(Q_H)\leq \tilde\rg(Q) \ \ , \ \ I=\inc(Q_H)\leq \tilde\inc(Q) \ \ , \ \ e = \rg(A) .$$
Our assumption implies $n-I-e \geq 1$.

 Set $(g_1 , \ldots , g_n) = (\rho(a_1), \ldots , \rho(a_n))$, where as above $\F_n = \langle  a_1,\ldots, a_n  \vert \emptyset \rangle$. Since $n\geq  \tilde\inc(Q)$, there is a Nielsen move that maps
 $(g_1, \ldots, g_n)$ to a tuple 
 \[
 (h_1, \ldots, h_{I}, h'_{I+1}, \ldots, h'_n)
 \]
where $\{ \bar h_1 , \ldots , \bar h_I \}$ generates $Q_H$ ($\bar h_j$ denoting the projection of $h_j$ in $Q$) and is incompressible (in particular $h_j \notin A$ for $1 \leq j \leq I$), and $h'_j \in A_H$ for $I+1 \leq j \leq n$. Let $B$ be the subgroup of $A_H$ generated by $\{h'_{I+1}, \ldots , h'_n \}$. The rank of $B$ is $\leq \rg(A_H)\leq \rg(A)=e$ thus we can fix a system of generators $\{ b_1, \ldots, b_{e} \}$ of $B$. 
Since $n-I \geq e +1 \geq d(B) +1$, Corollary~\ref{cor:lubotzky}\,(2) with $G=B$ gives a Nielsen move acting only on the $h'_j$ that maps
$(h'_{I+1}, \ldots, h'_n)$ to $(b_1, \ldots, b_e, 1, \ldots, 1)$, with at least one $1$ at the end. 
This shows that $\rho$ is redundant. \end{proof}

\begin{thm}\label{thm:jordan_size} 
 Let $G$ be a finite group with an exact sequence $ 0 \to A \to G \to Q\to 1$ where $A$ is an abelian group. 
 If $$ n \geq 1 + \rg(A) + \tilde\rg(Q)+\tilde\inc(Q) $$
 then the orbits of  $\Aut(\F_n)$ in $\Rep(\F_n; G)$ are exactly the subsets $\Epi(\F_n; H)$ where $H$ is any subgroup of $G$. 
\end{thm}

\begin{proof}
 Let $\rho\colon \F_n\to G$ be a homomorphism and let $H=\rho(\F_n)$. We denote by $Q_H$ the projection of $H$ in $Q$ and $A_H = A \cap H$.
Let $\{ q_1, \ldots q_d \} \subset H$ where $d = \rg(Q_H)$ and whose projection generates $Q_H$.
Let $\{ u_1, \ldots, u_e \} \subset A_H$ be a generating subset of $A_H$,  with  $e:=\rg(A_H)\leq \rg(A)$. 
Let $I = \inc(Q_H)$. As in the proof of Lemma~\ref{lem: exsequence}, we begin by performing a Nielsen move which maps $(g_1,\ldots,g_n)$ to $\omega_1 := (h_1, \ldots, h_I, h'_{I+1}, \ldots, h_n')$.  Let $B$ be the subgroup of $A_H$ generated by $\{h'_{I+1}, \ldots , h'_n \}$, it satisfies $\rg(B) \leq \rg(A_H) =e$. Let $\{ b_1, \ldots, b_{e} \}$ be a generating subset of $B$. 
Observe that  $$n-I \geq n - \tilde \inc(Q) \geq 1 + \rg(A) + \tilde \rg(Q) \geq 1 +e+ \rg(Q_H) .$$
In particular, $n-I \geq 1+e \geq 1 + \rg(B)$. By Corollay~\ref{cor:lubotzky}\,(2) applied to $B$, there is a Nielsen move which maps $\omega_1$ to 
$\omega_2 := (h_1, \ldots, h_I, b_1, \ldots, b_e , 1 , \ldots , 1)$. Since $n-I-e \geq 1 + \rg(Q_H)$,   
$\omega_2$ can be mapped to $$(h_1, \ldots, h_I, b_1, \ldots, b_e , q_1 , \ldots , q_d , 1 , \ldots , 1)$$ and then to   
$$\omega_3 := (\tilde h_1, \ldots, \tilde h_I, b_1, \ldots, b_e , q_1 , \ldots , q_d , 1 , \ldots , 1), $$  where $\tilde h_j \in A_H$ for $1 \leq j \leq I$.
Now observe that $n-I \geq 1+e \geq 1+ \rg(B')$ for every subgroup $B'$ of $A_H$, and that the following subset of $H$ generates $A_H$:
$$\{ \tilde h_1, \ldots, \tilde h_I, b_1, \ldots, b_e , q_1 , \ldots , q_d \}.$$
By using successively Corollary~\ref{cor:lubotzky}\,(2) and the occurence of the identity element $1$ in $\omega_3$, we can map $\omega_3$ to $$(\tilde h_1, \ldots, \tilde h_I, u_1, \ldots, u_e , q_1 , \ldots , q_d , 1 , \ldots , 1) , $$ and then to $( u_1, \ldots, u_e , q_1 , \ldots , q_d , 1 , \ldots , 1)$. This concludes the proof. 
\end{proof}

We say that a finite group $G$ has {\bf{Jordan size}} at most  $(J,R)$ ($J$ and $R$ being non-negative integers) if $G$ contains a normal, abelian subgroup $A$ of rank $\leq R$ such that $\card(G/A)\leq J$. Theorem \ref{thm:jordan_size} immediately implies:

\begin{cor}
\label{cor:epi_JR}
Let $G$ be a finite group of Jordan size at most $(J,R)$. If
$$ n \geq 1 + R + \max\{\tilde\rg(Q)+\tilde\inc(Q)\; ; \; Q {\text{ is a group with }} \card(Q) \leq  J\}$$ 
then the orbits of  $\Aut(\F_n)$ in $\Rep(\F_n; G)$ are exactly the subsets $\Epi(\F_n; H)$ where $H$ is any subgroup of $G$. 
\end{cor}

\section{Compact Lie groups}

\subsection{Peter-Weyl theorem} 
\label{par:peter-weyl}
One consequence of Peter-Weyl theory is that every compact Lie group admits a faithful continuous representation on a finite dimensional vector space. Moreover, any compact subgroup of $\GL_m(\R)$ is in fact an algebraic subgroup of $\GL_m(\R)$ (see~\cite{mneimne-testard}, p.78), and the connected components of the group (with respect to the euclidean topology) coincide with its irreducible components (for the Zariski topology). Altogether, we obtain the following theorem.
 
\begin{thm}[Peter-Weyl] If $G$ is a compact Lie group, there is an integer $m$, a real algebraic group $G'(\R)\subset\GL_m(\R)$ and an isomorphism of Lie groups $G\to G'(\R)$. 
\end{thm}

More precisely, there is an equivalence of category between compact Lie groups and compact real algebraic subgroups of $\GL_m(\R)$, every continuous homomorphism being automatically algebraic (since its graph is closed). 

Since every compact subgroup of $\GL_m(\R)$ is contained in a conjugate of the orthogonal group $\OO_m(\R)$, one can furthermore assume $G'\subset \OO_m(\R)$. Using complex representations and unitary groups instead of real representations, we get a similar result with $G'\subset \UU_m(\C)\subset\GL_m(\C)$ (doing so, the optimal value of $m$ can be smaller than for embeddings into orthogonal groups). In what follows, we use embeddings into $\GL_m(\C)$ (or $\UU_m(\C)$), {\emph{with $\GL_m(\C)$ considered as a real Lie group}}.

\subsection{Jordan theorem} 
The classical Jordan's theorem says that a finite subgroup of $\GL_m(\C)$ has Jordan size at most $(J(m),R(m))$ where $R(m)=m$ and $J(m)$ depends only on $m$, see~\cite{Breuillard:Jordan} or~\cite{Curtis-Reiner} for two proofs of this result. We shall always denote by $J(m)$ the optimal constant for which Jordan's theorem holds, and call it the {\bf{Jordan constant}}. 

\begin{rem}
The permutation group $\Bij_{m+1}$ embeds in $\GL_m(\R)$, so  $J(m)\geq (m+1)!$ for all $m$, and Collins proved in~\cite{Collins:Jordan} that 
$J(m)=(m+1)!$ as soon as $m\geq 71$ (he also computed $J(m)$ for $m\leq 70$). 
\end{rem}

We shall need the following version of Jordan's theorem, which is perhaps less known. If $G$ is a Lie group then $G^{\#}$ denotes the finite group of connected components of $G$: it is the quotient of $G$ by its neutral component $G^\circ$.

\begin{thm}[Jordan]\label{thm:jordan_general_version}
Let $m$ be a positive integer. Given any real algebraic subgroup $G$ of $\UU_m(\C)$ (resp.\ any complex algebraic subgroup of $\GL_m(\C)$) 
there is an exact sequence $0\to A\to G^{\#} \to Q\to 1$ where $A$ is abelian of rank $\leq m$ and $\card(Q)\leq J(m)$ (where  $J(m)$ is the Jordan constant).
\end{thm}

\begin{proof}
Let $\pi \colon G\to G^{\#}$ be the quotient map. Considering $\GL_m(\C)$ as a real algebraic group and $G$ as an algebraic subgroup, Theorem~1.1 of~\cite{Brion:Pacific2015} 
(see also Lemma 5.11 in~\cite{Borel-Serre:1964}) provides a finite subgroup  $F\subset G$ such that $\pi(F)=G^{\#}$. 

Applied  to $F$, the classical Jordan's theorem
 gives a normal subgroup $A_F\subset F$  such that  $\card(F/A_F)\leq J(m)$. 
Since $A_F$ is a finite abelian subgroup of $\GL_m(\C)$, it is diagonalizable, and it can be seen as a subgroup of  $(\C^\times)^m$. Such a group has rank $\leq m$, for every finite subgroup of  $(\R/\Z)^m$ is generated by $\leq m$ elements (every discrete subgroup of $\R^m$ has rank $\leq m$). To conclude, set $A = \pi(A_F)\subset G^{\#}$ and $Q = G^{\#} / A$. The group $A$ is normal in~$G^{\#}$, of rank $\leq m$, and of index $\leq J(m)$. 
\end{proof}

\subsection{Generators of finite groups} 

\begin{thm}[Kov\'acs and Robinson] 
\label{thm:kovacs-robinson}
Every finite subgroup of $\GL_m(\C)$ is generated by at most $\lfloor 3m/2\rfloor$ elements.
\end{thm}
In other words, $\rg(F)\leq \lfloor 3m/2\rfloor$ if $F\subset \GL_m(\C)$ is finite. This is proven in~\cite{kovacs-robinson}. 
Let 
\begin{equation}\label{eq:N_diese}
N_{\#}(m) =  1+  5m/2 + \log_2(J(m))  . 
\end{equation}

\begin{cor}
\label{cor:jordan_and_compact}
Let $G$ be a compact subgroup of $\GL_m(\C)$.
If $n\geq   N_{\#}(m)$, then
\begin{enumerate}[\rm (1)]
\item every $\rho\in \Rep(\F_n; G^{\#})$ is redundant, and
\item  the orbits of  $\Aut(\F_n)$ in $\Rep(\F_n; G^{\#})$ are exactly the subsets $\Epi(\F_n; H)$ where $H$ is any subgroup of $G^{\#}$. 
\end{enumerate}
\end{cor}

\begin{proof}
The group $G$ is identified with a subgroup of $\UU_m(\C)$. By Theorem \ref{thm:jordan_general_version}, there is an exact sequence $0\to A\to G^{\#} \to Q\to 1$ where $A$ is abelian of rank $\rg(A) \leq m$ and $\card(Q)\leq J(m)$.
The proof of Theorem \ref{thm:jordan_general_version} specifies that there exists a finite subgroup $F$ of $G$ with $\pi(F) = G^{\#}$. 
By Theorem~\ref{thm:kovacs-robinson}, any subgroup of $F$ is generated by $\lfloor 3m/2\rfloor$ elements.
The same property then holds for $\pi(F) = G^{\#}$ and for $Q$. Hence $\tilde\rg(Q)\leq \lfloor 3m/2\rfloor$ and $\rg(A) +  \tilde\rg(Q) \leq 5m/2$. Moreover, Inequality~\eqref{eq:d_ic_lc} implies $\tilde\inc(Q) \leq \log_2(\card (Q)) \leq  \log_2(J(m))$. We deduce 
$$ N_{\#}(m) =  1 +  5m/2 + \log_2(J(m)) \geq 1 + \rg(A) +  \tilde\rg(Q) + \tilde\inc(Q) .$$
The conclusion follows  by applying Lemma~\ref{lem: exsequence} and Theorem~\ref{thm:jordan_size} to $G^{\#}$.
\end{proof}

\begin{rem}
\label{rem:N_diese}
As a by-product of the preceding proof, we get $\rg(G^{\#}) \leq \lfloor 3m/2\rfloor$. This immediately implies $N_{\#}(m) \geq  2 +  \rg(G^{\#})$ for every $m \geq 1$, which will be used in the proof of Lemma \ref{lem:dense_in_epifin}.
\end{rem}

The constant $N_{\#}(m)\simeq m\log_2(m)$ is rather small. To get it we used~\cite{Collins:Jordan} and~\cite{kovacs-robinson}, which both rely on the classification of finite simple~groups, but an upper
bound $N_{\#}(m) \lesssim m^2\log_2(m)$ can be obtained from the  theorem of Schur described in~\cite{Curtis-Reiner}, page 258.

\begin{rem} Example~\ref{eg:zpr} shows that Corollary~\ref{cor:jordan_and_compact} is almost optimal: the best constant $N_\#(m)$ such that this corollary holds is necessarily $\geq m$.\end{rem}

\subsection{Generators of compact groups} \label{par:generators_of_compact_lie_groups}
Any connected, compact Lie group $G$ can be written as a quotient $(K\times T)/F$ where 
\begin{itemize}
\item $K$ is a product $K_1\times \cdots K_l$ of simply connected, almost simple, compact Lie groups $K_i$,
\item $T$ is a torus $\R^k/\Z^k$, and
\item $F$ is a finite subgroup of $K\times T$ that does contain any non-trivial element of $\{1\}\times T$.
\end{itemize}
A theorem of Kuranishi shows that  $K$ is topologically generated by $2$ elements (see~\cite{Kuranishi}). 
More precisely, there is an open and dense subset $U_K\subset K\times K$, of total Haar measure, such that every pair $(g,h)\in U_K$ generates a dense subgoup of $K$ (see Lemma~1.4 in~\cite{Gelander:Aut(Fn)2008}).
The following extension of these results will be used in Section \ref{12main}.

\begin{thm}
	\label{thm:top_generators}
	Let $G$ be a compact subgroup of $\GL_m(\C)$. Then, 
	\begin{enumerate}[\rm (1)]
\item $G$  is topologically generated by at most $\lfloor 2+(3m/2)\rfloor$ elements;
\item if $G$ is connected, pairs of elements generating $G$ form a subset of total Haar measure in $G\times G$.
\end{enumerate}
\end{thm}

\begin{proof} 
Write $G^\circ=(K\times T)/F$, as above.
Pick a pair $(g,h)$ in the open set $U_K$, choose $s, t\in T$ with $\langle s, t\rangle_{top}=T$, and consider the pair $S=((g,s), (h,t))$. 
The closed subgroup $H:=\langle S\rangle_{top}\subset K\times T$ projects onto $K$ under the first projection and onto $T$ under the second. At the level of Lie algebras, $\mathfrak{h}\subset \mathfrak{k}\oplus \mathfrak{t}$ projects surjectively onto $\mathfrak{k}$ and $\mathfrak{t}$. Since $\mathfrak{k}$ is semi-simple, this implies that $\mathfrak{h}= \mathfrak{k}\oplus \mathfrak{t}$. Hence $H=K\times T$. This proves that $G^\circ$ is topologically generated by $2$ elements; more precisely, the generating pairs form a dense subset of full Haar measure in $G^\circ \times G^\circ$. This proves the second assertion.

For the first assertion, we combine Theorem~1.1 of~\cite{Brion:Pacific2015}, which provides a finite subgroup  of $G$ intersecting every connected component of $G$, with Theorem~\ref{thm:kovacs-robinson}, and with the second assertion (applied to $G^\circ$).
\end{proof}

In this proof, if we replace the reference to~Theorem~\ref{thm:kovacs-robinson} by Jordan's theorem, we obtain a similar statement  with $2+m+J(m)$ in place of
$\lfloor 2+(3m/2)\rfloor$. 

\section{A redundancy Theorem}

\begin{thm}\label{thm: rdu}
For every  integer $m \geq 1$, there exist a positive integer $N(m)$ such that 
$ \Rep (\F_n ; \UU_m(\C)) \subset \Red (\F_n ; \UU_m(\C)) $ for all  $n\geq N(m)$.
Equivalently, 
for every compact subgroup $G$ of $\GL_m(\C)$ and every $n\geq N(m)$,  we have 
$$  \Epitop (\F_n ; G) \subset \Redtop (\F_n ; G). $$
\end{thm}

 The proof is given in Section~\ref{par:proof:redundancy} below. It shows that
 \begin{equation}
 \label{eq:N_m}
 N(m)=2m N_\#(m)=2m (1+5m/2+\log_2(J(m)))
 \end{equation}
 works, but we do not claim this is optimal. 
 
\begin{rem}
\label{rem:redundant_vs_k-redundant}
Theorem~\ref{thm: rdu} implies that every homomorphism $\F_n\to \UU_m(\C)$ is $k$-redundant if $n\geq N(m)+k-1$.
\end{rem}

In parallel to the proof of Theorem~\ref{thm: rdu}, we shall also prove a version of the theorem in the category of linear algebraic groups. In this context, the notion of epimorphism and redundant representation
have to be considered with respect to the Zariski topology. Then, the theorem becomes the following.

\begin{thm}\label{thm: rdu_alg}
Let $ \bfk $ be a field of characteristic $0$. 
For every  integer $m \geq 1$, there exist an integer $Z(m)$ such that 
$ \Rep (\F_n ; \GL_{m}(\bfk) ) \subset \Red (\F_n ; \GL_{m}(\bfk)) $ for all $n\geq Z(m)$.
Equivalently, 
for every algebraic subgroup $G$ of $\GL_m(\bfk)$ and every $n\geq Z(m)$,  we have 
$  \Epitop (\F_n ; G) \subset \Redtop (\F_n ; G). $
\end{thm}

As we shall see, we can take $Z(m)=\frac{1}{2}m(m+5)N_\#(m)$. This answers positively Conjectures 4.2 and 4.3 of~\cite{Gelander:JournalAlgebra2024}, but   only when $n\geq Z(m)$.

\begin{rem}
\label{rem:Cohen-Vigdorovich}
In~\cite{Cohen-Vigdorovich}, the authors obtain a similar redundancy statement: {\emph{let $G\subset \GL_m(\C)$ be a simple complex linear group;
 if $(g_1, \ldots, g_n)$ generates a Zariski dense subgroup of $G$ and $n\geq a m^{10}$, for some universal constant $a>0$, 
then $(g_1, \ldots, g_n)$ is compressible}}.  Their proof relies on three main ingredients. Suppose, for simplicity, that $G\subset \GL_m(\C)$ is defined over $\Z$ and the $g_i$ are in $G(\Z)$. First, the $g_i$ generate $G(\Fp)$ for all large primes, by Weisfeiler's theorem. Second, 
the kernel of $G(\Z_p)\to G(\Fp)$ is the Frattini subgroup of $G(\Z_p)$, again for all large primes. Third, a recent theorem of Harper~\cite{Harper:2023} implies that $\inc(G(\Fp))\leq a m^{10}$ for some universal $a>0$.
\end{rem}

\subsection{Redundancy} 

\begin{lem}\label{lem:redundancy_from_partial_redundancy}
Let $G$ be a topological group. 
Suppose $\rho\colon \F_n\to G$ is a homomorphism and $\F_n=A\star B$ is a free decomposition such that $\rho_{\vert A}$ is redundant. 
Then, $\rho$ is redundant. 
\end{lem}

Let us simultaneously prove this result and rephrase it in terms of Nielsen
moves. We start with $(g_1, \ldots, g_n)\in G^n$.
The assumption means that,
up to some Nielsen move,
we have
$ \langle g_2, \ldots, g_{j}\rangle_{top}=  \langle g_1, \ldots, g_{j}\rangle_{top}$ 
for some $j\geq 1$. To show that $(g_1, \ldots, g_n)$ is redundant, it suffices to show
$\langle g_2, \ldots, g_n\rangle_{top}
=
\langle g_1, \ldots, g_n\rangle_{top}
$.
But $\grouptop{g_2, \dots , g_n}$ contains $\grouptop{g_2, \dots, g_j}$, hence also $g_1$; so it contains  
$ \left\{ g_1, \dots , g_n \right\} $, hence also $\grouptop{ g_{1}, \dots , g_n  }$.

\begin{rem}
This proof applies to   algebraic groups with their Zariski topology, even though the product topology on $G^n$ does
not coincide with the Zariski topology.
\end{rem}

\subsection{Proof of Theorem~\ref{thm: rdu}}\label{par:proof:redundancy}

Let $m \geq 1$ be fixed.

\subsubsection{Chains of connected subgroups} 
For a compact Lie group $G$, define $\ell(G)$
to be the longest length of an increasing sequence of closed and connected subgroups $G_0=\{1_G\}\subset G_1\subset \cdots \subset G_\ell=G$. Thus,  $\ell(G)=0$ if and only if $G$ is finite. Note that $\ell(H)\leq \ell(G)$ if $H\subset G$ and $\ell(G)=\ell(G^\circ)$. 

In the unitary group $\UU_m$, consider the sequence $G_1=\UU_1$, 
$G_2=\UU_1\times \UU_1$, $\ldots$, $G_m=\UU_1^m$ where the copies of $\UU_1$ are along the diagonal. Then,  define $G_{m+j}$=$\UU_j\times \UU_1^{m-j}$, up to $G_{2m-1}=\UU_m$. This construction shows that $\ell(\UU_m)\geq 2m-1$. In fact, we have equality, as the following theorem shows.

\begin{thm}
\label{thm:longest_chain_in_Um}
If $G$ is a closed subgroup of $\UU_m$, then $\ell(G)\leq \ell(\UU_m)=2m-1$.
\end{thm}

This result follows from  Theorems 1 and 2 in~\cite{burness-liebeck-shalev} (see also~\cite{onishchik-vinberg;book} for the study of maximal subalgebras of semisimple Lie algebras). It replaces the obvious quadratic upper bound $\ell(\UU_m)\leq \dim_\R(\UU_m)=m^2$ by a linear bound. 

\begin{rem}
Let $G$ be a linear algebraic group, defined over some algebraically closed field $\bfk$. Denote by $\ell_a(G)$ the maximal length of an increasing chain of connected algebraic subgroups $G_i\subset G$. Let $R_u(G)$ be the unipotent radical. The quotient ${\overline{G}}:=G/R_u(G)$ is reductive, and we denote by ${\mathsf{rk}}({\overline{G}})$ the rank of this group and by $B({\overline{G}})$ its Borel subgroup. Then, Theorem~1 and~3 of~\cite{burness-liebeck-shalev-alg} show that
\[
\ell_a(G)= \dim(R_u(G)) + {\mathsf{rk}}({\overline{G}}) + \dim(B({\overline{G}})) > \dim(G)/2
\]
This gives 
$\ell_a(\GL_m(\bfk)=\frac{1}{2}m(m+3)$, which is quadratic in $m$. 
Thus, $\ell_a(G)\leq \frac{1}{2}m(m+3)$ for any algebraic subgroup of $\GL_m(K)$ over a field $K$ of characteristic $0$.
\end{rem}

\subsubsection{Definition of $N(m)$} 
Define recursively the sequence $N(m,D)$  by
\begin{align}
N(m,0)&=N_\#(m) \label{eq:definition_N(m,D)_II}\\
N(m, D+1) &= N(m,0) +  \max_{0 \leq j \leq D}  N(m,j)\label{eq:definition_N(m,D)_II}
\end{align}
Note that $N(m,D)$ increases strictly with $D$ and   $m$. 
Thus, in Equation~\eqref{eq:definition_N(m,D)_II}, the max is equal to $N(m,D)$, which yields $N(m,D)=(D+1)N(m,0)$. In particular,  if we define $N(m):=N(m, 2m-1)$ we obtain
\begin{equation}\label{eq:definition_N(m)}
N(m) =  2m N_\#(m).
\end{equation}

\subsubsection{Description of the  recursion} 
Let $P(\ell)$ be the assertion: {\textit{for every $n \geq N(m,\ell)$ and for every compact subgroup $G\subset \GL_m(\C)$ with $\ell(G)\leq \ell$, any homomorphism $\rho : \F_n \to G$ is redundant. }}

We shall prove $P(\ell)$ by induction on $\ell$. 
Then, thanks to Theorem~\ref{thm:longest_chain_in_Um},  Theorem \ref{thm: rdu} follows with $N(m)$ as in Equation~\eqref{eq:definition_N(m)}.

\begin{rem}
To prove Theorem~\ref{thm: rdu_alg}, we set $Z(m)=N(m,m(m+3)/2)=\frac{1}{2}m(m+5)N_\#(m)$.
\end{rem}

\subsubsection{Proof of the  recursion} To verify Property $P(0)$ we only have to refer to~Corollary~\ref{cor:jordan_and_compact}.

Now, assume that $P(k)$ is satisfied for $0\leq k \leq \ell$. To prove $P(\ell+1)$, we fix some $n\geq N(m,\ell+1)$. We take $G$ with $\ell(G)=\ell+1$ and a homomorphism $\rho\colon \F_n\to G$, and  we set $g_i=\rho(a_i)$, $1\leq i\leq n$. Replacing $G$ by the closure of the image of $\rho$, we can assume that  $\rho \in \Epitop(\F_n;G)$.
 
Let $\pi : G \to G^{\#}$ denote the quotient morphism.  
Since $n \geq N(m,0)$,  Corollary~\ref{cor:jordan_and_compact} tells us that every element of $\Rep(\F_n;G^{\#})$ is redundant.
Thus, there is a Nielsen move (mapping the $g_i$ to some $g_i'$) and some integer 
$k\leq N(m,0)-1$  such that $(\pi(g_1'), \ldots, \pi(g_k'))$ generates $G^{\#}$; from this, we can also impose that 
$g_{k+1}'$, $\ldots$, $g_n'$ are all in $G^\circ$,  the neutral component of $G$. Now, we set 
\[
H=\langle g_{k+1}', \ldots, g_n' \rangle_{top}
\]
and distinguish two cases. 

The first case is when $\ell(H) < \ell( G)$. 
By definition of $N(m,\ell+1)$, we have  $n-k \geq  1+ \max_{0 \leq j \leq \ell} N(m,j)$, which implies  by induction that $(g_{k+1}', \ldots, g_n')$ 
 is redundant. By~Lemma~\ref{lem:redundancy_from_partial_redundancy}, $\rho$ is also redundant.

In the second case   $\ell (H) = \ell (G)$, or equivalently $H = G^\circ$.
Let $K\subset H$ be the closed subgroup defined by $K= \langle  g'_{k+1} , \ldots ,  g'_{n-1}  \rangle_{top}$. 
If  $\ell(K) = \ell(G)$, then $K = H$ and $(g'_{k+1}, \ldots, g'_n)$ is immediately redundant. 
Otherwise $\ell (K) < \ell (G)$. Since 
$n \geq N(m,0) + N(m , \ell(K) )$, we obtain $n-k-1 \geq N(m , \ell(K) )$ and the induction hypothesis shows that $(g'_{k+1}, \ldots, g'_{n-1})$ 
is redundant. 
In both cases, we conclude with Lemma~\ref{lem:redundancy_from_partial_redundancy}. 

\begin{rem}
The proof of Theorem~\ref{thm: rdu_alg} is the same, except that the recursion goes up to $\ell=\ell_a(\GL_m(\bfk))$ instead of $\ell(\UU_m)$. 
\end{rem}

\section{Dynamics}

Here, we apply the results of the previous section
to prove
Theorem \ref{thm:main}.

\subsection{Orbit closures}
\label{sec:orbit_closure}
In this section we classify the orbit closures
for the action of $ \Aut(\F_n) $
on
$ \Rep(\F_n; G) $,
when $ n $ is large.

\subsubsection{Two lemmas} Let $ k $ be an integer $< n $.
Recall that
  a homomorphism $\rho\colon \F_n\to G$
is  $k$-redundant
if there is a decomposition
$\F_n=A\star B$
such that (i) $ B $ is a free group of rank at least $ k $
and (ii)
 $\rho(A)$ and $\rho(\F_n)$ have the same closure.
Recall that we denote by  $\Redtop^{k} (\F_n ; G)$  the set of these $k$-redundant homomorphisms.
The following lemma is the topological version
of Theorem~\ref{thm:lubotzky} (1).
\begin{lem}
	[Minimality on $ d $-redundant representations]
	\label{lem:red_implies_dense}
	Let $d$ and $n$ be two integers such that $0\leq d\leq n$ and $n\geq 2$. 
	Let $G$ be a topological group which is topologically generated by $ d $
	elements.
	If
	$\rho \in \Redtop^{d} (\F_n; G) \cap \Epitop(\F_n; G)$, then
		 the $ \Aut(F_n) $-orbit of
	$ \rho $
	is
	dense in
	$\Redtop^{d} (\F_n; G) \cap \Epitop(\F_n; G)$.
\end{lem}

Note that $G$ needs not be a Lie group for this lemma.
For application to the Zariski topology,
we do not specify the topology on
$ \Rep(\F_n; G) = G^{n} $;
we only require that the projections are continuous.
To prove this lemma, we shall use repetitively the following two facts. 

\noindent{ (a)} For a continuous group action, denote by $x\supset y$ the fact that the orbit of $y$ is contained in the orbit closure of $x$. Then, $\supset$ is transitive: if $x\supset y$ and $y\supset z$, then $x\supset z$.

\noindent{ (b)} If
	$ (g_1, \dots, g_n) \in G^{n} $
	and $ h \in \langle g_1, \dots, g_{n-1}\rangle$
	then
	there are Nielsen moves going from
	$ (g_1, \dots, g_n)$
	to
	$ (g_1, \dots, g_{n-1},  g_n h )$.
		If, instead, 
	$ h \in \langle g_1, \dots, g_{n-1}\rangle_{top}$
	then
	$(g_1, \dots, g_{n-1},  g_n h)$
	is in the orbit closure of
	$ (g_1, \dots, g_n)$.
\begin{proof}
	Set $ m = n - d $. Fix a tuple $ (g_1, \dots, g_d) $ that generates $ G $ topologically.
	
	Set	$(h_1 \ldots , h_n ) =	(\rho(a_1) , \ldots , \rho(a_n))$.
	As $ \rho $ is $ d $-redundant
	and an epimorphism,
	up to applying Nielsen moves,
	we can assume that
	$ \langle h_1 \ldots , h_m \rangle_{top} = G$.
	
	We first show that
	$ (g_1, \dots, g_d, 1, \dots, 1) $
	is in the orbit closure of
	$ \rho $.
	For $ k > m $,
	we have $ h_k \in
	\langle h_1 \ldots , h_m \rangle_{top}$,
	hence
	$(h_1, \dots, h_m, 1, \dots, 1)$ is in the orbit closure of
	$ (h_1, \dots, h_n) $.
	By the same argument,
	\[
		(h_1, \dots, h_m, 1, \dots, 1)
		\supset
		(h_1, \dots, h_m, g_1, \dots, g_d)
		.
	\]
		Now using that
	$ (g_1, \dots, g_d) $ generates $ G $ 
	topologically, we see that
	\[
		(h_1, \dots, h_m, g_1, \dots, g_d)
		\supset
		(g_1, \dots, g_d, 1, \dots, 1)
		.
	\]

	Consider now another representation
	$ \rho' \in
	\Redtop^{d} (\F_n; G) \cap \Epitop(\F_n; G)$
	and set
	$(h'_1 \ldots , h'_n ) =
	(\rho'(a_1) , \ldots , \rho'(a_n))$.
	As before we can assume that
	$ (h'_1, \dots, h'_m) $ generates $ G $ topologically.
	The same reasoning as above, in reverse,
	gives
	\begin{align*}
		(g_1, \dots, g_d, 1, \dots, 1)
		&\supset
		(g_1, \dots, g_d, h'_1, \dots, h'_m)\\
		&\supset
		(h'_1, \dots, h'_m, 1, \dots, 1)\\
		&\supset
		(h'_1, \dots, h'_n)
		.
	\end{align*}
	(Note that we used that $\rho'$ is an epimorphism on the second line.)
	We conclude that the orbit of
	$ \rho $ is dense in
	$\Redtop^{d} (\F_n; G) \cap \Epitop(\F_n; G)$.
\end{proof}

\begin{lem}
	\label{lem:dense_in_epifin}
	Let $ G $ be a compact subgroup of
	$ \GL_m (\C) $.
	Let $ n $
	be an integer such that
	$ n \geq N_\#(m)$.
	Then the closure of
	$\Epitop(\F_n; G)$
	is
	$ \Epifin(\F_n; G) $.
\end{lem}
\begin{proof}
	It suffices to show that
	$\Epitop(\F_n; G)$
	is dense in 
	$ \Epifin(\F_n; G) $,
	because
	$\Epitop(\F_n; G)
	\subset
	\Epifin(\F_n; G) $
	and
	$ \Epifin(\F_n; G) $
	is closed.

	When $ G $ is connected,
	the assertion means
	that
	$\Epitop(\F_n; G)$
	is a dense subset of 
	$\Rep(\F_n; G) $,
	which follows from
	Lemma~1.10
	of~\cite{Gelander:Aut(Fn)2008}.
	In this case
	$ n \ge 2 $
	is sufficient; this will be used below.

	For a general $ G $,
	fix  $ \rho \in \Epifin(\F_n; G) $
	and set
	$ \left( g_{1}, \dots , g_{n}   \right)
	=
	\left( \rho(a_1), \dots, \rho(a_n) \right)$.
	Since 
	$\Redtop^{d} (\F_n; G)$, $\Epitop(\F_n; G)$,
	and
	$ \Epifin(\F_n; G) $
	are $ \Aut(\F_n) $-invariant,
	we can argue up to Nielsen moves.
	Since $n\geq N_\#(m)$, we know by Remark~\ref{rem:N_diese} that $n\geq d(G^\#)+2$.  
	Now, with Corollary~\ref{cor:jordan_and_compact}, we can assume, after some  Nielsen move,
         that 	$ (g_{1}, \dots, g_{n-2}) $
	generates $ G^{\#} $. Applying another Nielsen move, we can also assume that
	$ g_{n-1}$ and $g_{n} $
	are in~$ G^{o} $.
	
	Using the case of a connected group,
	we can find a sequence
	$  (g_{n-1}^{(k)} , g_{n}^{(k)})
	\in
	\Epitop(\F_2; G^{o})$
	which converges to
	$ (g_{n-1}, g_{n}) $.
	Then 
	$ (g_{1}, \dots, g_{n-2},
	g_{n-1}^{(k)} , g_{n}^{(k)})$
	is a sequence in
	$\Epitop(\F_n; G)$
	that converges to
	$ \left( g_{1}, \dots , g_{n}   \right)$.
\end{proof}

\subsubsection{Proof of Assertions~(1) and~(2) of Theorem~\ref{thm:main}} 
\label{12main}
Let $ G $ be a compact Lie group. 
Let $ m $ be the smallest integer $ \ge 1 $ such that $ G $ embeds in $ \GL_{m}(\C) $.
Set $ d = 2 + \left\lfloor 3m/2 \right\rfloor $.
Fix $ n \ge N(m) + d $, where $ N(m) $ is the constant given in Theorem \ref{thm: rdu} (see Equation~\eqref{eq:N_m}).

Fix $ \rho \in \Rep(\F_n; G) $
and denote by
$ H_{\rho} $
the closure of the image of $ \rho $.
By definition,
$ \rho \in \Epitop(\F_n; H_{\rho}) $.
By Theorem \ref{thm: rdu} and
Remark~\ref{rem:redundant_vs_k-redundant},
\[ \Epitop(\F_n; H_{\rho}) =
\Epitop(\F_n; H_{\rho})
\cap
\Redtop^{d}(\F_n; H_{\rho}).
\]
In particular,
the representation
$ \rho $
is $ d $-redundant.

By Theorem \ref{thm:top_generators}\,(1),
the group $ H_{\rho} $ can be topologically generated
by $ d $ elements.
Thus, applying
Lemma \ref{lem:red_implies_dense} to $ \rho $,  the
$ \Aut(\F_n) $-orbit of $ \rho $ is dense in
$ \Epitop(\F_n; H_{\rho}) $.
Since, 
the closure of 
$ \Epitop(\F_n; H_{\rho}) $
is
$ \Epifin(\F_n; H_{\rho}) $ (by 
Lemma \ref{lem:dense_in_epifin}),
we conclude that the
orbit closure of $ \rho $
is
$ \Epifin(F_n; H_{\rho}) $.

\subsection{Invariant probability measures}
Here, we prove Assertion~(3) of
Theorem~\ref{thm:main}.


\subsubsection{Disintegration of measures}\label{par:disintegration} Let us 
recall the concept of disintegration of a probability measure with respect to a fibration 
(see~\cite{Einsiedler-Ward}, Theorem 5.14).
	
Let $ (X,\mu)$ and $(Y,\bar{\mu}) $ be Borel probability spaces,
with a measurable map 	$ \pi : X \to Y $ such that
	$ \bar{\mu} = \pi_{*} \mu $.
	Then there exists a measurable
	family of probability measures
	$ \mu_{y} $ for $ y \in Y $
	such that
	$$
	\mu
	=
	\int_{Y}
	\mu_{y}\, 
	d \bar{\mu}(y)
	,
	$$ 
	and
	$ \bar{\mu} $-almost every
	$ \mu_y $ is supported on
	$ \pi^{-1}(y) $.
	Moreover,
	if
	$ T : X \to X $
	is a measurable transformation
	such that
	$ \mu $ and $ \pi $ are invariant
	under the action of~$ T $,
	then
	for
	$ \bar{\mu} $-almost every $ y $,
	the measure
	$ \mu_y $ is $ T $-invariant.

\subsubsection{Closed subgroups}\label{par:closed subgroups} 

\begin{thm}\label{par:countably_many_subgroups}
Let $G$ be a compact Lie  group. Then $G$ contains at most countably many conjugacy classes of closed subgroups.
\end{thm}

\begin{proof}[Sketch of proof]    
The first main ingredient is a simple theorem of Montgomery and Zippin, which says that, for any compact subgroup $F$ of $G$, there is a neighborhood $U$ of $F$ in $G$ with the following property: if $H$ is a subgoup of $G$ contained in $U$, then $g^{-1}Hg\subset F$ for some $g\in G$ (see Section~5.3 p. 215 of~\cite{Montgomery-Zippin}). Moreover, if $H$ is isomorphic to $F$ (as a Lie group), then $g^{-1}Hg=F$, because 
$g^{-1}H^\circ g=F^\circ$ since the dimensions are the same; and then $g^{-1}Hg=F$ since $\card(H^{\#})=\card(F^{\#})$ (see~Lemma 2.1 in~\cite{Lee-Wu}).

We already described the second ingredient in Section~\ref{par:generators_of_compact_lie_groups}:  up to isomorphism, there are only countably many compact Lie groups, hence only countably many possibilities for the compact subgroups in $G$

Now fix such a compact Lie group $F$ and consider the space $\Rep(F,G)$ of homomorphisms from $F$ to $G$. Each homomorphism $F\to G$ is algebraic with respect to the natural structures of real algebraic groups of $F$ and $G$ described in Section~\ref{par:peter-weyl}, so $\Rep(F,G)$ is a countable union of algebraic varieties (homomorphisms given by formulas of degree $\leq d$ form an affine variety). As such, it has at most countably many connected components. The property that two homomorphisms 
are conjugate is a closed property, and by the theorem of Montgomery and Zippin it is also open. This concludes the proof. 
\end{proof}

\subsubsection{Supports of invariant probability measures} 

\begin{lem} \label{lem:support}
	Let $ G $ be a compact Lie group
		and
	$ \mu $ be an
	$ \Aut(\F_n) $-invariant and ergodic probability
	measure on $ \Rep(\F_n; G) $.
	Then there exists a closed subgroup
	$ H $ of $ G $ 
	such that $ \mu $
	is supported on
	$ \Epitop(\F_n; H) $.
\end{lem}
\begin{proof}
{\textit{Step 1.--}} {\textit{There is a closed subgroup $H$ 
such that $\mu$ is supported by the set of morphisms $\rho\colon \F_n\to G$ 
such that $\overline{\rho(\F_n)}$ is a conjugate of~$H$. }}
			
According to Theorem~\ref{par:countably_many_subgroups}, we can fix   a countable family  $ (H_i)_{i \in I} $ of closed subgroups, with exactly one $H_i$ per conjugacy class.
	
For subsets $ A \subset G $and $ X \subset \Rep(\F_n; G) $, we set 
$ X^{A} = \left\{ g \rho g^{-1}; \, g \in A, \, \rho \in X \right\} $.
Remark that $ X^{A} $ is compact if $ X $
and $ A $ are compact.
	
Fix a closed subgroup $ H $. 
Let $ I_{H} \subset I $
be the set of indices $ i $ corresponding to conjugacy classes (in G) of proper closed subgroups of $H$.
We have 
\[
\Epitop(\F_n; H)^{G}
=
\Rep(\F_n; H)^{G}
\setminus
\bigcup_{i \in I_H} 
\Rep(\F_n; H_i)^{G}
\]
(here we use, as in the proof of Theorem~\ref{par:countably_many_subgroups}, that a closed subgroup cannot contain strictly one of its conjugates, see Lemma 2.1 of~\cite{Lee-Wu}). Hence
	$ \Epitop(\F_n; H)^{G} $ 
	is measurable,
	for
	each
	$ \Rep(\F_n; H_i)^{G} $
	is compact.
	We have
	$$
	\Rep(\F_n; G)
	=
	\bigsqcup_{i} 
	\Epitop(\F_n; H_i)^{G},
	$$
	where the union is disjoint.
	As each
	$ \Epitop(\F_n; H_i)^{G}  $ 
	is
	$ \Aut(\F_n) $-invariant,
	the ergodicity  implies that $ \mu $ gives mass $1$ to
	$ \Epitop(\F_n; H)^{G} $
	for  a unique closed subgroup $ H $ of $ G $.

	{\textit{Step 2.--}} {\textit{ The measure 
	$ \mu $ is supported
	on a unique 
	conjugate of
	$ \Epitop(F_n; H )$.}}
	
	For $ g_1$ and $ g_2 \in G $,
	the two subsets
	$ \Epitop(F_n; H )^{g_1}$
	and
	$\Epitop(F_n; H )^{g_2}$
	intersect
	if and only if
	they are equal,
	and this happens
	exactly when
	$ g_{2}^{-1} g_1 $ 
	is in
	$ {\mathrm{N}}(H) $, 	the normalizer of $ H $
	in $ G $.
	In other words,
	$ \Epitop(\F_n; H)^{G} $
	is partitioned by the subsets
	$ \Epitop(F_n; H)^{g} $
	for $ g \in G / {\mathrm{N}}(H) $.
	Denote by $ p $ the map
	from
	$ \Epitop(\F_n; H)^{G} $
	to
	$ G / {\mathrm{N}}(H) $
	which maps a representation
	$ \rho $ 
	to the unique $ g \in G / {\mathrm{N}}(H) $
	such that
	$ \rho \in
	\Epitop(F_n; H)^{g} $.
	It is a measurable  map (\footnote{Indeed, 
	let $ \bar{A} $
	be a compact subset of
	$ G / {\mathrm{N}}(H) $,
	and $ A $ its preimage in $ G $.
	The inverse image
	$ p^{-1}(\bar{A}) $
	is equal 
	$ \Epitop(\F_n; H)^{A} $.
	As before, 	
	$
	\Epitop(\F_n; H)^{A}$
	coincides with
	$ \Rep(\F_n; H)^{A}
	\setminus
	\bigcup_{i \in I_H} 
	\Rep(\F_n; H_i)^{G},
	$
	hence this subset is measurable.
	This implies that $ p $ is measurable.}).
	And it is $\Aut(\F_n)$-invariant. Thus, by ergodicity, $\mu$ is supported on a single fiber of $p$.
\end{proof}

\subsubsection{Unique ergodicity on redundant representations} 
\begin{lem}\label{lem:ergodic_redundant}
Let $G$ be a compact subgroup of $\GL_m(\C)$. Suppose $n\geq  N_\#(m)$. 
Let $\mu$ be an $\Aut(\F_n)$-invariant, ergodic probability measure on $\Rep(\F_n;G)$.
If $\mu$ gives mass $1$ to the set 	$ \Redtop^4(\F_n; G) \cap \Epitop(\F_n; G) $, then 
$ \mu $ is the algebraic measure~$ \mes_{G}^n $.
\end{lem}

Before starting the proof, as in \S~\ref{par:intro_main_result}, consider the natural map 
$$\# \colon \Rep(\F_n; G) \mapsto \Rep(\F_n; G^\#);$$
 to a
representation $\rho\colon \F_n\to G$, it associates the representation $\rho^\#\colon \F_n\to G^\#$ 
obtained by composition with the projection $\pi\colon G\to G^\#$. 
For each element $g^\#$ in $G^\#$, we fix an element $\hat{g}\in G$ with $\pi(\hat{g})=g^\#$ and assume that  
the lift of the neutral element $1^\#$ of $G^\#$ is the neutral element $1\in G$. 
Then, given any element $(g^\#)=(g_1^\#, \ldots, g_n^\#)$ of $(G^\#)^n=\Rep(\F_n; G^\#)$, the fiber $F=\#^{-1}(g^\#)$ can be identified to $\prod_{i=1}^n \hat{g_i} G^\circ$. This gives a map 
$$(h_1, \ldots, h_n)\in (G^\circ)^n \mapsto  (\hat{g_1}h_1, \ldots, \hat{g_n}h_n)$$
parametrizing $F$. The image of the Haar measure $\mes_{G^\circ}^n$ under this map is, by definition, the measure $\mes_F$. As explained in \S~\ref{par:intro_main_result}, this measure $\mes_F$ does not depend on the preliminary choices of lifts $\hat{g}$. And if $\varphi\in \Aut(\F_n)$ maps a fiber $F$ of $\#$ to a fiber $F'$, then $\varphi_*(\mes_F)=\mes_{F'}$. Moreover, 
\begin{equation}
\label{eq:measure_mGn}
\mes_G^n=\frac{1}{\vert \Epi(\F_n;G^\#) \vert}\sum_{(g^\#)\in \Epi(\F_n; G^\#)} \mes_{F_{(g^\#)}}
\end{equation}
where $F_{(g^\#)}$ is the fiber of $\#$ above $(g^\#)$.

\begin{proof}[Proof of Lemma~\ref{lem:ergodic_redundant}] Denote by $\pi_{1, n-4}\colon G^n\to G^{n-4}$ the projection to the first $n-4$ factors, and by $\pi_{n-1,n}$ the projection on the last two factors; similarly, let $\pi_{n-3,n-2}$ be the projection $(g_1, \ldots, g_n)\mapsto (g_{n-3},g_{n-2})$. 

\vspace{0.1cm}

{\sl{Step 1.--}}
 Since $n\geq  N_\#(m)$, we know from Corollary~\ref{cor:jordan_and_compact} that $\Aut(\F_n)$ acts transitively on $\Epi(\F_n;G^\#)$. Thus, the projection of $\mu$ on 
$\Epi(\F_n;G^\#)$ is the counting measure. All we have to prove is that the restriction $\mu_F$ is proportional to $\mes_F$ for every fiber $F$ of $\#$. 

\vspace{0.1cm}

{\sl{Step 2.--}} Let
	$ R = \pi_{1, n-4}^{-1}(
	\Epitop(\F_{n-4}, G))$ be the set of
		$ (g_1, \dots, g_n) $
	such that
	$ (g_1, \dots, g_{n-4}) $
	generates a dense subgroup of $G$.
	Then
	$$
	\Redtop^4(\F_n; G) \cap \Epitop(\F_n; G)
	=
	\bigcup_{\varphi \in \Aut(\F_n)} \varphi(R)
	.
	$$
	The measure $ \mu ( \varphi(R) ) = \mu(R) $
	is independent of $ \varphi \in \Aut(\F_n)$ and 
	by assumption
	$ \mu
	(
	\Redtop^4(\F_n; G) \cap \Epitop(\F_n; G)
	)
	= 1 $,
	hence
	$ \mu(R) > 0 $.
	Since $\Epi(\F_n; G^\#)$ is finite, at least one fiber $F$ of $\#$ satisfies 
	$$\mu(R\cap F)>0.$$
	We fix such a fiber $F$ and denote by $(g_1^\#, \ldots, g_n^\#)=\#(F)$ its image in $\Epi(\F_n; G^\#)$.
	By definition of $R$, $(g_1^\#, \ldots, g_{n-2}^\#)$ generates $G^\#$. Thus, applying an automorphism $\varphi$ of $\F_n$ that fixes the first $n-4$ generators  $a_1, \ldots, a_{n-4}$, we can assume that $g_{j}^\#=1^\#$ for $j=n-3, n-2, n-1, n$. Doing so, we obtain a new fiber $F'$ of $\#$, namely the one above $(g_1^\#, \ldots, g_{n-4}^\#, 1^\#, \ldots, 1^\#)$, which still satisfies $\mu(R\cap F')>0$. For simplicity, we denote $F'$ by $F$ and assume $g_{j}^\#=1^\#$ for $j\geq n-3$.

\vspace{0.1cm}

{\sl{Step 3.--}} Let us prove that the Haar measure $\mes_F$ is ergodic under the action of the stabilizer of $F$ in $\Aut(\F_n)$. This is a variation on a theorem of Gelander (see~\cite{Gelander:Aut(Fn)2008}).

For this, set $\Aut(\F_n)_F$ to be the stabilizer of $F$ (equivalently, of $\#(F)$ in $\Epi(\F_n; G^\#)$), and suppose $A\subset F$ is an  $\Aut(\F_n)_F$-invariant and measurable subset of positive Haar measure. Then, its projection $\pi_{n-1,n}(A)$ is a measurable subset of $(G^\circ)^2$ of positive Haar measure. By Theorem~\ref{thm:top_generators}\,(2), almost every pair $(h_{n-1}, h_n)$ in $\pi_{n-1,n}(A)$ generates a dense subgroup of $G^\circ$. Now, using Nielsen moves of type 
$$
(\hat{g}_1 h_1, \ldots, \hat{g}_{n-2} h_{n-2}, h_{n-1}, h_n)\mapsto (\hat{g}_1 h_1 w_1 , \ldots, \hat{g}_{n-2} h_{n-2}w_{n-2}, h_{n-1}, h_n)
$$
where each $w_j$ is a word in $ h_{n-1}$ and $h_n$, we see that for almost every element in~$A$, the whole fiber of $\pi_{n-1, n}$ through that element is almost entirely contained in $A$. We shall say that $A$ is saturated with respect to the projection $\pi_{n-1, n}$.

Then, consider the projection $\pi_{n-3,n-2}$. By what we  
proved, $\pi_{n-3,n-2}(A)$ is a subset of total Haar measure in $(G^\circ)^2$. Now, we consider moves of type 
$$
(\hat{g}_1 h_1, \ldots, \hat{g}_{n-2} h_{n-2}, h_{n-1}, h_n)\mapsto (\hat{g}_1 h_1  , \ldots, \hat{g}_{n-2} h_{n-2}, h_{n-1}w_{n-1}, h_nw_n),
$$
where $w_{n-1}$ and $w_n$ are words in $h_{n-3}$ and $h_{n-2}$ (recall that $\hat g_{n-3}=1_G=\hat g_{n-2}$); then,  the same argument shows that~$A$ is saturated with respect to the projection $\pi_{1, n-2}$. 
These two saturation properties imply that $A$ is in fact a subset of total Haar measure. Thus $\mes_F$ is ergodic. 

As a consequence, $\mes_G^n$ is ergodic for the action of $\Aut(\F_n)$, because $\Aut(\F_n)$ permutes transitively the fibers of $\#$.

\vspace{0.1cm}

{\sl{Step 4.--}}  Let us come back to the study of $\mu_F$. By Steps 1 and 2, $\mu_F(R\cap F)>0$. 
Let us disintegrate $\mu_F$ with respect to the projection $\pi_{1, n-2}$. The conditional measures $\mu_{(g)}$, for 
$$(g)\in \prod_{i=1}^{n-2}\hat{g_i}G^\circ,$$
are probability measures on $(G^\circ)^2$. Set $S:=\pi_{1, n-2}(R\cap F)$. For $(g)$ in $S$, the measure $\mu_{(g)}$
is invariant by all right translations 
$$
(h_{n-1}, h_n)\mapsto (h_{n-1} w_{n-1}, h_{n} w_n)
$$
where $w_{n-1}$ and $w_n$ are any words in $(g)=(\hat{g_i}h_i)$ such that $w_{n-1}(g^\#)=1^\#=w_n(g^\#)$.
For $(g)$ in $S$, these pairs of words form a dense subset of $(G^\circ)^2$. Thus, $\mu_{(g)}$ coincides with the Haar measure $\mes_{G^\circ}^2$. Since $\mu_F(R\cap F)>0$, we conclude that 
$$(\pi_{n-1,n})_*\mu_F \geq \alpha \mes_{G^\circ}^2$$ 
for some $\alpha   \geq \mu_F(R\cap F)> 0$. By this we mean that if  $B\subset (G^\circ)^2$ is  measurable, then  $(\pi_{n-1,n})_*\mu_F(B)\geq \alpha \mes_{G^\circ}^2(B)$. 

Now, for a set $P\subset (G^\circ)^2$ of total measure for $\mes_{G^\circ}^2$, each pair $(h_{n-1},h_n)\in P$ 
generate a dense subgroup of $G^\circ$. Using right multiplication by words in $(h_{n-1}, h_n)$ on the fibers of
$\pi_{n-1,n}\colon F\to (G^\circ)^2$, we deduce that the conditional measures $\mu_{(h_{n-1}, h_n)}$ of $\mu_F$ 
are Haar measures (more precisely, are images of the Haar measure $\mes_{G^\circ}^{n-2}$ by the parametrization of $F$). Thus, $\mu_F\geq \alpha' \mes_F$ for some $\alpha' > 0$. By invariance of $\mu$ and $\mes_G^n$ under the action of $\Aut(\F_n)$, this property remains true on every fiber of $\#$. Thus, $\mu\geq \alpha' \mes_G^n$. 
And by ergodicity of $\mu$ and $\mes_G^n$, we deduce that $\mu=\mes_G^n$.
\end{proof}

\subsubsection{Conclusion} 
We now show Assertion (3) of
Theorem \ref{thm:main}.
Let $ G $ be a compact Lie group. 
Let $ m $ be the smallest integer $ \ge 1 $ such that $ G $ embeds in $ \GL_{m}(\C) $. 
Fix $n\geq N(m) + 3$. 
By Lemma~\ref{lem:support}, $\mu$ is supported on $\Epi(\F_n; H)$ for some compact group $H\subset G\subset\GL_m(\C)$.
By Theorem~\ref{thm: rdu}, $\Epi(\F_n; H)\subset \Red^4(\F_n;H)$.
Since $n\geq N(m)$, we have $n\geq N_\#(m)$ and 
Lemma~\ref{lem:ergodic_redundant} can be applied to the Lie group $H$.  Hence $\mu$ is the algebraic measure $\mes_H^n$.

\begin{rem} 
\label{rem:estimation_Nm}
In Section~\ref{12main} we used the lower bound $N(m)+2 + \left\lfloor 3m/2 \right\rfloor $, here we used $N(m)+3$, which is smaller. From Equation~\eqref{eq:N_m}, Collins' theorems in~\cite{Collins:Jordan}, and Stirling's formula, this is bounded from above by 
$ b + 3m^2  + 2m^2\log_2(m) $ for some constant $b\leq10^3$ (which can be replaced by $0$ for $m\geq 71$). This justifies the Inequality~\eqref{eq:N(m)_intro}.
\end{rem}

\section{Applications}
\label{par:Applications}

\subsection{Character varieties} 
As in Section~\ref{par:peter-weyl}, we endow the compact Lie group $G$ with its natural real algebraic structure.

Let $\pi\colon \Rep(\F_n; G)\to \chi(\F_n;G)$ be the projection onto the character variety that is, on the 
quotient $\chi(\F_n;G)=  \Rep(\F_n; G)/\!\!\!/ G$ for the action of $G$ by conjugacy. Here, the quotient is 
taken in the sense of geometric invariant theory, so that $\pi$ is a morphism of algebraic varieties, hence continuous. By compactness of $G$, the fibers of $\pi$ are orbits of the conjugacy action on $\Rep(\F_n; G)$.

The action of $\Aut(\F_n)$ on $\Rep(\F_n; G)$ induces an action on $\chi(\F_n; G)$ which factors through the group $\Out(\F_n)=\Aut(\F_n)/\Inn(\F_n)$.

If $H$ is a compact group its Haar measure $\mes_H$ is invariant under the action of the group of continuous automorphisms $\Aut(H)$. Thus, if $H$ is a closed subgroup of $G$, $\mes_H$  is invariant under the action of
the normalizer of~$H$ in~$G$. Now, consider the  probability measure $\mes_H^n$ defined in the Introduction, the support of which is $\Epifin(\F_n;H)$, and project it to get a probability measure 
$
\pi_*\mes_H^n
$
on $\chi(\F_n;G)$;  this measure depends only on the conjugacy class $[H]$ of $H$ in $G$. Doing so, we obtain 
at most countably many distinct probability measures $\mes_{[H]}^n:=\pi_*\mes_H^n$. Theorem~\ref{thm:main} directly implies the following statement.

\begin{thm}
\label{thm:dynamics_on_character_bounded}
Let $G$ be a compact subgroup of $\GL_m(\C)$. If $n\geq N(m)+2+\lfloor\frac{3m}{2} \rfloor$,
the $\Out(\F_n)$-invariant and ergodic probability measures in $\chi(\F_n;G)$ are exactly the measures $\mes_{[H]}^n$. If $K$ is a compact invariant subset on which the action of $\Out(\F_n)$ is topologically transitive, then $K=\pi(\Epifin(\F_n;H))$ for a unique closed subgroup $H$ of $G$. 
\end{thm}

\subsection{Compact orbits in non-compact groups}%
\label{sec:non-compact}
In this section, we replace $G$ by the non-compact linear group $\SL_m(\R)$, for some $m\geq 1$. 
We still denote by $\pi\colon \Rep(\F_n;\SL_m(\R))\to \chi(\F_n;\SL_m(\R))$ the projection on the character variety. 
Unlike the compact case, the fibers
are not conjugacy classes in general.
Two representations are identified in the character variety
when
the closures of their conjugacy classes intersect.

A representation~$\rho\colon \F_n\to \SL_m(\R)$ is {\textit{reducible}} if its image preserves a proper
non-trivial subspace of $\R^m$; otherwise, it is {\textit{irreducible}}. 
It is \emph{completely reducible}
if it preserves each factor of a decomposition
$ \R^{m} = \bigoplus V_{i} $
and its restriction to each factor is irreducible.

If $ \rho $ is reducible,
we can find a
completely reducible representation $ \rho_0 $
such that $ \rho $ and $ \rho_{0} $
have the same image in the character variety.
By Jordan-Hölder theory, it is unique up to conjugacy.
We call $ \rho_{0} $ the
\emph{semisimplification}
of $ \rho $.
To construct it, 
assume that $\rho$ preserves a non-trivial and proper subspace $V\subset \R^m$. 
Let $W$ be any subspace such that $\R^m=V\oplus W$. Let $r_t$ be the linear transformation of $\R^m$
acting by multiplication by $t^w$ on $V$ and by $t^{-v}$ on $W$ with $v=\dim(V)$ and $w=\dim(W)$,
so that  $r_t$ is in $\SL_m(\R)$.
If one conjugates $\rho$ by $r_t$ and let $t$ go to $+\infty$, then at the limit 
one gets a new representation $\rho'$ that preserves each factor of the direct sum $V\oplus W$.
Repeating this process at most $m$ times, we see that there is a completely
reducible representation $\rho_0$ in the closure of the conjugacy class of
$ \rho $. As the projection $ \pi $ is continuous,
$ \rho $ and $ \rho_{0} $ have the same image in the character variety.

\begin{thm}
\label{thm:boundedness_in_charac_var}
	Let $ m $ be a positive integer. 
	Let  $ n $ be  large enough.
	Let $\rho$ be an element of  $ \Rep(\F_n; \SL_m(\R)) $.
	If the
	$ \Out(\F_n) $-orbit of
	$ \pi(\rho) $ 
	is relatively compact in
	$ \chi(\F_n; \SL_m(\R)) $,
		then the semisimplification
	$\rho_0 $
	of $ \rho $ 
	takes values in a compact group. 
	If moreover 
	$ \rho $ is completely reducible,
	$ \rho $ itself takes value in a compact group.
\end{thm}

For reducible representations, passing to  the semisimplification is necessary.
For example, consider a non-trivial representation $\rho$ into the group of unipotent and  upper-triangular matrices.
Its image in $ \chi(\F_n; \SL_m(\R)) $ is the class of the trivial representation, hence it is fixed by $ \Out(\F_n) $,
but $\rho$ is not  bounded. 

\begin{cor}
\label{cor:minimal}
	For $ m \ge 1 $ and $ n $ large enough,
	the minimal $ \Out(\F_n) $-invariant compact subsets
	in
	$ \chi(\F_n; \SL_m(\R)) $
	are the sets
	$ \pi(\Epifin(\F_n; G)) $,
	where $ G $ is a compact subgroup
	of $ \SL_m(\R) $. 
\end{cor}

If one considers $\C^m$ as a real vector space, one gets an embedding of $\SL_m(\C)$ into $\SL_{2m}(\R)$ and a continuous map $\chi(\F_n; \SL_m(\C))\to \chi(\F_n; \SL_{2m}(\R))$. This can be used to prove the above statements for $\SL_m(\C)$ in place of $\SL_m(\R)$.  

In particular, if $n$ is large ($n\geq  Z(2m)$) and if $\rho\colon \F_n\to \SL_m(\C)$ has a finite orbit in the character variety, then the image of its semisimplification is finite.

\begin{proof}[Proof of Theorem~\ref{thm:boundedness_in_charac_var}] 
         Let
	$ G $ be the (real) Zariski closure of
	$ \rho (\F_n)$ in $\SL_m(\R)$.

	\vspace{0.1cm}

{\textit{Step 1.--}} First, we assume that $\rho$ is irreducible.	
	If $G$ is compact, it is conjugate to a subgroup of the maximal compact subgroup $\SO_m(\R)$, 
	and we are done. So, arguing by contradiction, we 	suppose that $ G $ is not compact.
	Denote by $ r $ the \emph{proximal rank} of $ G $
	(see \cite[Section 4.1]{benoistRandomWalksReductive2016}).
	By definition, this is the smallest integer $ \ge 1 $
	for which there exist elements $ g_k \in G $, 
	scalars $ \lambda_k \in \R$, and a non-zero endomorphism $f$ of rank $ r $ such that
	$ \lambda_k g_k $ converges towards~$ f $.
	As $ G $ is not compact,
	the rank is $ \leq m-1 $
	and $ f $ is not invertible.
	
	By Theorem \ref{thm: rdu_alg},
	$\rho$ is redundant with respect to the Zariski topology.
	We can suppose that
	$ \rho(F_{n-1}) = \left< \rho(a_1), \dots, \rho(a_{n-1})\right> $
	is
	Zariski dense in
	$ G $.
	By
	\cite[Lemma 6.23]{benoistRandomWalksReductive2016}
	the proximal rank of $ \rho(\F_{n-1}) $ is also
	equal to $ r $.
	Thus, there is a sequence
	$ w_k \in \F_{n-1}  $
	and scalars $ \lambda_k $ such that
	$ \lambda_k \rho(w_k) $ converges to some  
	endomorphism
	$ f $ of rank $ r $.
	As the determinant of
	$ \rho(w_k) $ is $1$
	and $ f $ is not invertible,
	$ \lambda_{k} $
	converges towards $0$.
	By irreducibility
	of $ \rho(F_{n-1}) $,
	up to replacing $ f $
	by $ \rho(w) f $
	for some $ w \in F_{n-1} $,
	we can assume that
	$ f $ is not nilpotent,
	for otherwise,
	we would have
	$ \rho(w) \mathrm{Im} f \subset \mathrm{Ker} f $
	for every $ w \in F_{n-1} $,
	which would contradict the irreducibility of
	$ \rho(F_{n-1}) $.

	Let $ h = \rho(a_n) $.
	Take $ w_1$, $w_2 \in \F_{n-1} $,
	to be determined later.
	There exists a sequence
		$ \varphi_{k} \in \Aut(\F_n) $
	such that
	$ (\varphi_{k})_{*} \rho(a_{n})
	= \rho(w_1 a_n w_2 w_k) $.
	Since
	\begin{align*}
		\lambda_k
			(\varphi_{k})_{*}  \rho(a_{n})
		=
		\lambda_k
			\rho(w_1 a_n w_2 w_{k})
		=
			\rho(w_1)
			h
			\rho(w_2)
			\lambda_k
			\rho(w_k)
			\end{align*}
we obtain 
\begin{align*}
		\lim_{k\to +\infty} \lambda_k
			(\varphi_{k})_{*}  \rho(a_{n})
&=			\rho(w_1)
			h
			\rho(w_2)
			f.
	\end{align*}
	Taking traces, we see that if 
		$ \tr
	\left( 
		\rho(w_1)
		h
		\rho(w_2)
		f
	\right) \neq 0 $,
	then
	$ \tr 
	\left( 
		(\varphi_{k})_{*}  \rho(a_{n})
	\right) $
	goes to $ +\infty $
	with $ k $.
	As the function
	$ \rho \mapsto \tr ( \rho(a_n) ) $
	induces a continuous function
	on the character variety,
	this would imply that the orbit of
	$ \pi(\rho) $
	is not relatively compact in
	$ \chi(\F_n; \SL_m(\R)) $.

	Let us now show that we can find
	$ w_1$ and $ w_2 \in \F_{n-1} $
	such that
	\[ \tr
	\left( 
		\rho(w_1)
		h
		\rho(w_2)
		f
	\right) \neq 0.
	\]
	Indeed, otherwise 
	$ \tr
	\left( 
		g_1
		h
		g_2
		f
	\right) = 0$
	for every
	$ g_1$,  $g_2 \in \rho(\F_{n-1}) $.
	As this is a polynomial equation
	in $ g_1 $ and $ g_2 $,
	and  
	$ \rho(\F_{n-1}) $
	is Zariski-dense in $ G $,
	it is  satisfied
	by all 	$ g_1$, $g_2 \in G $. 
	Choosing   $ g_1 = h^{-1} $, we obtain
	$ \tr
	\left( 
		g
		f
	\right) = 0$
	for every $g\in G$.
	Choosing $ g = \rho(w_k)^{l-1} $ with $l\in \Z$,
	this gives
	$ \tr
	\left( 
		( \lambda_k \rho(w_k))^{l-1} 
		f
	\right) = 0$.
	When $ k $ goes to infinity,
	we obtain
	$ \tr
	\left( 
		f^{l} 
	\right) = 0$,
	for every $l\in \Z$.
	This implies that
	$ f $ is nilpotent,
	a contradiction.

	\vspace{0.1cm}
{\textit{Step 2.--}} Now, we do not assume anymore that  $\rho$ is irreducible.
Consider the semisimplification  $ \rho_{0} $ of $ \rho $.
Let $\R^m=\oplus_i V_i$ be a decomposition of $\R^m$ in irreducible representations of $\rho_0$. 
By assumption, $\pi(\rho_0)$ has a bounded orbit. This means that the regular functions on $\Rep(\F_n; \SL_m(\R))$ which are invariant under conjugacy are bounded along $\Aut(\F_n)(\rho_0)$. These functions include all symmetric functions in the eigenvalues of $\rho(w)$, for  any $w\in \F_n$; and there is a finite list of words $(w_j)$, with $w_i=a_i$ for $i \leq n$, such that  the functions $\rho\mapsto \tr(\rho(w_j))$ generate the algebra of regular function on $\chi(\F_n; \SL_m(\R))$.
Moreover,  a subset of~$\C^m$ is bounded if and only if  its image under the elementary symmetric functions is also bounded.
From these remarks, we see that the boundedness of $\Out(\F_n)(\pi(\rho_0))$  is equivalent to the following: all eigenvalues of all elements $\rho(\varphi(w_j))$ have modulus $\leq D$, for all $\varphi\in \Aut(\F_n)$ and some $D>0$.
As a consequence, for each $i$, the point in $\chi(\F_n;\SL(V_i))$ determined by the restriction $(\rho_0)_{\vert V_i}$ has a bounded orbit under the action of $\Out(\F_n)$. By the first step, the restriction of $\rho_0(\F_n)$ to each $V_i$ is relatively compact, and $\rho_0(\F_n)$ itself is relatively compact,
hence takes values in a compact subgroup.
\end{proof}

\subsection{Constraints on primitive elements}

Let $ P_n \subset \F_n $ be the set of primitive elements, where $a\in \F_n$ is {\textit{primitive}} if it is an element of a free basis of $\F_n$.

Let $ d(m) < m^2 $ be the smallest integer such that every Zariski closed subgroup
of $ \SL_m(\C) $ can be topologically generated by
$ d(m) $ elements.

\begin{lem}
\label{lem:primitive_dense}
Let $ \rho$ be a homomorphism from $\F_n$ to $\SL_m(\C) $
and let $ G $ be the Zariski closure of  $ \rho (\F_n)$.
If $n\geq N(m)+d(m)-1$, then $ \rho(P_n) $ is Zariski dense
in $ G $.
\end{lem}

\begin{proof}  By
Theorem~\ref{thm: rdu_alg},
$ \rho $ is $ d(m) $-redundant for the Zariski topology.
By Lemma~\ref{lem:red_implies_dense},
the orbit of $ \rho $ is Zariski dense
in 
$\Redtop^{d} (\F_n; G) \cap \Epitop(\F_n; G)$.
Now,
$ \rho(P_n) = \pi_1( \Aut(\F_n)\rho) $,
where $ \pi_1$ is the projection on the first coordinate,
so
the closure of $ \rho(P_n) $
contains
$ \pi_1(
\Redtop^{d} (\F_n; G) \cap \Epitop(\F_n; G)) = G$.
\end{proof}

The following statement is precisely Conjecture 7.1 (or equivalently 7.3) of~\cite{Gelander:JournalAlgebra2024} {\textit{when $n$ is sufficiently large}}. 

\begin{thm}
Let ${\mathrm{Uni}}(m)\subset \GL_m(\C)$ be the algebraic subvariety of unipotent elements. 
Suppose $n\geq N(m)+d(m)-1$. 
If $ \rho \in \Rep(\F_n; \GL_m(\C))$ maps
every primitive element to a unipotent element, {\sl{i.e.}}\ $\rho(P_n)\subset {\mathrm{Uni}}(m)$,
then  $\rho(\F_n)$ is conjugate to a group of unipotent upper triangular matrices. 
\end{thm}

Indeed, Lemma~\ref{lem:primitive_dense} implies that every element of $\rho(\F_n)$ is unipotent. 
The theorem follows, because a subgroup of $\GL_m(\C)$
whose elements are all unipotent is conjugate to a group of upper triangular matrices. 

\section{Redundancy for surface groups}

Let $ S_{g} $ be a closed surface of genus $ g $, and let $\pi_1(S_g)$ denote its fundamental group. 
The group $\Aut(\pi_1(S_g))$ acts on $\Rep(\pi_1(S_g); G)$ for any group $G$, and one might wonder whether an analogue of Theorem~\ref{thm:main} holds in this context. We don't know yet how to obtain such a statement. 
But we provide an analogue of Theorem~\ref{THMB}.

Following Dunfield-Thurston,
we say that
$ \rho : \pi_1(S_g) \to G $
is a \emph{stabilization} if
$ S_{g} $ can be written as a connected sum
$ S_{g} = S_{g_1} \# S_{g_2} $
and
if there exists a representation
$ \rho' : \pi_1(S_{g_1} ) \to G $
such that $ \rho $ is equal to
$ \rho' $
on
$ S_{g_1} $
and the trivial representation on
$ S_{g_2} $ (in particular, $\rho(\gamma)=1_G$ on the curve along which the sum $S_{g_1} \# S_{g_2} $ is done). The following is Proposition 6.16 of~\cite{Dunfield-Thurston}.

\begin{thm}
\label{thm:dunfield-thurston}
	Let $ G $ be a finite group.
	If $ g > \card (G) $
	then every homomorphism
	$ \rho : \pi_1(S_g) \to G $
	is a stabilization.
\end{thm}

We establish the following generalization.

\begin{thm}
\label{thm:dunfield-thurston-uniform}
	Let $ G $ be a finite group
	of Jordan size at most
	$ (J,R) $.
	If $ g \ge J + R + 1 $
	then every
	$ \rho : \pi_1(S_g) \to G $
	is a stabilization.
\end{thm}

\begin{cor}
\label{cor:dunfield-thurston-uniform}
For every  integer $m \geq 1$, there exist a positive integer $T(m)$ such that 
$ \Rep (\pi_1(S_g) ; \UU_m(\C)) \subset \Red (\pi_1(S_g) ; \UU_m(\C)) $ for all  $g\geq T(m)$.
\end{cor}

Here, $\rho$ is redundant if there is a decomposition $ S_{g} = S_{g_1} \# S_{g_2} $ such that $\rho(\pi_1(S_{g_1}^*))$ is dense in $\rho(\pi_1(S_{g}))$, where $S_{g_1}^*$ is the punctured surface.
Once Theorem~\ref{thm:dunfield-thurston-uniform} is at our disposal, the corollary can be derived exactly as  Theorem~\ref{thm: rdu}. Indeed, the only place in this proof where Nielsen moves are used in a non-trivial way is to deal with the discrete part $G^\#$ of $G$. The constant $T(m)$ that we get is $2m(1+ m + J(m))$.

\begin{proof}[Proof of Theorem~\ref{thm:dunfield-thurston-uniform}]
	Let
	$ 0\to A \to G \to Q\to 1$
	be an
	exact sequence
	for $ G $
	with $ A $ abelian of rank $ \leq R $
	and $ Q $ of cardinal $\leq  J $.
	
	{\textit{Step 1.--}} Composing $ \rho $
	with the quotient map
	$ G \to Q $
	we obtain
	$ \bar{\rho} : \pi_1(S_g) \to Q $.
	Applied to
	$ \bar{\rho} $,
Theorem~\ref{thm:dunfield-thurston} 	gives a decomposition
	\[ S_{g} = S_{g-1} \# T \]
	which induces a decomposition
	of the fundamental group
	$ \pi_1(S_{g}) =
	\pi_1(S_{g-1}^{*})
	\star_{\gamma} 
	\pi_1(T^{*})
	$,
	where $ T^{*} $ is a one-holed torus
	and
	$ S_{g-1}^{*} $ is a one-holed
	genus $ g-1 $ surface,
	such that
	the restriction of $ \bar{\rho} $
	to
	$\pi_1(T^{*})$
	is trivial.
	This means that
	the restriction of $ \rho $
	to
	$\pi_1(T^{*})$
	takes values in the abelian group $ A $.
	In particular, the curve $ \gamma $ along which
	$ S_{g-1} $ and $ T $ are attached,
	which is equal to the commutator of the generators
	of
	$\pi_1(T^{*})$,
	is mapped by $ \rho $ to the identity in $ G $.
	We deduce that
	$ \rho $
	induces two representations
	$ \rho : \pi_1(S_{g-1}) \to G $
	and
	$ \rho : \pi_1(T) \to A $.

	The same process can be applied to
	$ \rho : \pi_1(S_{g-1}) \to G $.
	Iterating this process $ R+1 $ times,
	we obtain a decomposition 
	$ S_{g} = S_{g-R-1} \# S_{R+1} $
	such that $ \rho $
	induces representations
	$ \rho : \pi_1(S_{g-R-1}) \to G $
	and
	$ \rho : \pi_1(S_{R+1}) \to A $.

	{\textit{Step 2.--}} It is now sufficient to restrict the study to the abelian part: 
	we can assume that
	$ \rho : \pi_1(S_{g}) \to A $
	takes values in an abelian group
	and $ g \ge R+1 $. The representation factorizes through the abelianization
	and can be identified with a vector $ v \in H^{1}(S_{g}; A) $ (i.e.\ $v \in \Rep(S_g; A)$).
	Via the choice of a symplectic basis,
	we identify
	$ H^{1}(S_{g}; A) $
	with
	$ A^{2g} $.
	The action of the mapping class group
	$ \Mod_{g} $ 
	of $ S_{g} $
	on
	$ H^{1}(S_{g}; A)$
	factorizes through the
	action of
	$ \Sp_{2g}(\Z) $
	on
	$ A^{2g} $.
	Recall that the morphism
	$\Mod_{g} \to \Sp_{2g}(\Z) $
	is surjective
	(see Theorem~6.44 in~\cite{FarbMargalit}).
	We will show that there exists
	$ M \in \Sp_{2g}(\Z) $
	such that the last two coordinates
	of
	$ Mv $ are null.
	This will imply that there exists
	a one-holed torus in $ S_{g} $
	in restriction to which
	$ \rho $ is trivial.

	\begin{lem}
		The symplectic group
		$ \Sp_{2g}(\Z) $
		acts transitively
		on primitive vectors
		in $ \Z^{2g} $.
	\end{lem}
	\begin{proof}[Proof (see Lemma~5.3 in~\cite{Benoist:Integral_Symplectic})]
	Let $(u_1, \ldots, u_g, v_1, \ldots, v_g)$ be the ca\-nonical symplectic basis of $\Z^{2g}$,
	the symplectic form being given by $u_1^*\wedge v_1^* + \ldots +u_g^*\wedge v_g^*$.
	Let $w$ be a primitive vector in $\Z^{2g}$. If $g=1$, there is an element of 
	$\Sp_2(\Z)=\SL_2(\Z)$ that maps $w$ to $u_1$ (by Bézout's theorem). Thus, we can assume
	that $w=\sum_i a_i u_i$ for some primitive vector $w_0=(a_1, \ldots, a_g)\in \Z^{g}$. Since 
	$\SL_g(\Z)$ acts transitively on primitive vectors, there is a matrix $M\in \SL_g(\Z)$ such that
	$Mw_0=u_1$; the matrix ${\mathrm{diag}}(M, ^t\!M^{-1})$ is in $\Sp_{2g}(\Z)$ and maps $w$ to~$u_1$.
	\end{proof}
	
	Let us decompose $ A $
	as a sum of cyclic groups
	$ A =  A_{1} \oplus \cdots \oplus A_{r} $ with $r\leq R$,
	and  denote by
	$ p_{i} $ the  projection onto $A_i$.
	We can write $ v $ as a $ 2g \times R $ matrix:
	$$
	\begin{bmatrix}
		p_1(v_1) & \cdots & p_{R} (v_1) \\
		\vdots & \ddots & \vdots \\
		p_1(v_{2g} ) & \cdots & p_{R} (v_{2g})
	\end{bmatrix}
	$$

	Lifting $ p_1(v) \in A_1^{2g}  $ to $ \Z^{2g}  $,
	 the lemma provides
	$ M_1 \in \Sp_{2g}(\Z) $
	such that $ M_1 p_1(v) = (\star,\star, 0, \dots, 0) $,
	hence $ M_1 v $ can be written as a block matrix:
	$$
	\left[
	\begin{array}{c|ccc}
		\star & \star & \cdots & \star \\
		\star & \star & \cdots & \star \\
		\hline
		0 & \star & \star & \star \\
		\vdots &    \vdots & \vdots &  \vdots \\
		0 & \star & \star & \star
	\end{array}
	\right]
	$$
	Using an element of
	$ \Sp_{2g}(\Z) $ of the form
	$ \begin{bmatrix}
		\Id_2 & 0 \\
		0 & M_2 
	\end{bmatrix} $,
	we can apply the same process
	to the bottom-right matrix,
	which is in
	$ (A_{2} \oplus \cdots \oplus A_{R})^{2(g-1)} $.
	After $ r $ iteration,
	$ v $ is sent
	by an element of
	$ \Sp_{2g}(\Z) $
	to a vector with its last two coordinates null.
\end{proof}

\bibliographystyle{plain}

\bibliography{referencesRedon}

\end{document}